\documentclass[11pt]{amsproc}
\usepackage{eurosym}
\usepackage{amssymb}
\usepackage{amsfonts}
\usepackage{cmtt}
\usepackage{tikz,pgfplots}
\usepackage{mathdots}
\usetikzlibrary{arrows.meta}
\usepackage{xcolor}
\usepackage[paper=a4paper, margin=2.5cm, tmargin = 3cm]{geometry}
\usepackage{hyperref}
\hypersetup{
    unicode=false,
    pdftoolbar=true,
    pdfmenubar=true,
    pdffitwindow=false,
    pdfstartview={FitH},
    pdftitle={},
    pdfauthor={},
    pdfkeywords= {}
    pdfnewwindow=true,
    colorlinks=true,
    linkcolor=blue,
    citecolor=red,
    filecolor=magenta,
    urlcolor=cyan
}
\usepackage[capitalize]{cleveref}
\usepackage[ruled,linesnumbered, inoutnumbered, algosection]{algorithm2e}
\usepackage[format=plain, font=footnotesize]{caption}
\theoremstyle{plain}
\newtheorem{theorem}{Theorem}[section]

\newtheorem{corollary}[theorem]{Corollary}
\newtheorem{lemma}[theorem]{Lemma}
\newtheorem{proposition}[theorem]{Proposition}

\theoremstyle{definition}
\newtheorem*{acknowledgement}{Acknowledgement}

\newtheorem{definition}[theorem]{Definition}
\newtheorem{example}[theorem]{Example}
\newtheorem{remark}[theorem]{Remark}

\usepackage[backend=bibtex,maxbibnames = 99]{biblatex}
\theoremstyle{remark}
\numberwithin{equation}{section}
\newcommand{\field}[1]{\mathbb{#1}}
\newcommand{\C}{\field{C}}

\newcommand{\N}{\field{N}}

\newcommand{\R}{\field{R}}
\newcommand{\Z}{\field{Z}}

\newcommand{\p}{\field{P}}

\renewcommand{\theta}{h}

\newcommand{\vol}{{\mathrm{vol}}}
\newcommand{\Conv}{{\mathrm{Conv}}}

\newcommand{\superscript}[1]{\ensuremath{^{\textrm{#1}}}}

\begin{document}

\title{The Algebraic Degree of Coupled Oscillators}

\author[Paul Breiding]{Paul Breiding\superscript{\,a}\footnote{\superscript{a} pbreiding@uni-osnabrueck.de; University of Osnabr\"uck, Albrechtstr.\ 28a, 49076 Osnabrück, Germany. Supported by DFG, German Research Foundation -- Projektnummer 445466444.}}

\author[Mateusz Micha{\l}ek]{Mateusz Micha{\l}ek\superscript{b}\footnote{\superscript{b} mateusz.michalek@uni-konstanz.de; University of Konstanz, Fach D 197, 78457 Konstanz, Germany. Supported by DFG, German Research Foundation -- Projektnummer 46757530.}}

\author[Leonid Monin]{Leonid Monin\superscript{c}\footnote{\superscript{c} leonid.monin@mis.mpg.de; Max Planck Institute for Mathematics in the Sciences, Inselstra\ss e 22, 04103 Leipzig, Germany.}}

\author[Simon Telen]{Simon Telen\superscript{d}\footnote{\superscript{d} simon.telen@cwi.nl; Centrum Wiskunde \& Informatica (CWI), Science Park 123, 1098 XG Amsterdam, Netherlands. Supported by a Veni grant from the Netherlands Organisation for Scientific Research (NWO).}}


\begin{abstract}
Approximating periodic solutions to the coupled Duffing equations amounts to solving a system of polynomial equations. The number of complex solutions measures the algebraic complexity of this approximation problem. Using the theory of Khovanskii bases, we show that this number is given by the volume of a certain polytope. We also show how to compute all solutions using numerical nonlinear algebra.
\end{abstract}

\maketitle


\section{Introduction} \label{sec:1}

Counting the number of solutions to a system of nonlinear equations is a ubiquitous problem in mathematics and its applications. For instance, this may provide bounds on the complexity of an algorithm, or a proof of completeness for a set of computed solutions. Typically, equations seen in applications are highly structured, and naive root counts such as B\'ezout's theorem largely overestimate the actual number. Moreover, standard techniques count solutions over algebraically closed fields, such as $\C$, although applications often care about other fields, like $\R$. Apart from counting solutions, one might be interested in identifying parameter regions with a large (or small) number of real solutions, and in actually computing these solutions. Recent theory and algorithms in computational algebraic geometry can be used as a powerful tool to analyse and solve structured equations. Our aim in this paper is to develop such an analysis for equations coming from the study of limit cycles in physics. At the same time, we aspire to illustrate how the used techniques can be applied in a more general context. Key words include discriminants, Khovanskii bases and homotopy continuation.

We briefly sketch our setting. The starting point is a single oscillator, described by an ordinary differential equation in the unknown displacement function $t \mapsto X(t) \in \R$:
\begin{equation} \label{eq:duffing}
\ddot{X} + \alpha  \, X + \beta \, X^3 - \gamma \, \cos(\omega t) + \delta \, \dot{X} \, =  \, 0.
\end{equation}
This is known as the \emph{Duffing equation}, which is a standard example of a nonlinear dynamical system showing chaotic behaviour for certain ranges of the parameters $\alpha, \beta, \gamma, \delta, \omega$. We are interested in the non-chaotic regime where there is a periodic steady state solution $X(t) = X(t + 2 \pi k), k \in \Z$. This is inspired by research in physics; see  \cite{kovsata2022harmonicbalance} and references therein.

\emph{Coupled} Duffing oscillators are described by the following system of differential equations:
\begin{equation} \label{eq:coupledduff}
\ddot{X}_i + \alpha  \, X_i + \beta \, X_i^3 - \gamma \, \cos(\omega_i t) + \delta \, \dot{X}_i \,  + \sum_{\ell \neq i} J_{i\ell} \, X_\ell =  \, 0, \quad i = 1, \ldots, N.
\end{equation}
Here the unknown functions are $X_1(t), \ldots, X_N(t)$. There is one frequency parameter $\omega_i$ for each~$i$, and $J_{i \ell}$ is a new matrix of parameters representing the \emph{coupling} between our $N$ oscillators. Again, the periodic steady states of this dynamical system are of interest. 

The method of \emph{harmonic balancing} remarkably turns the problem of approximating periodic solutions to \eqref{eq:coupledduff} into a system of algebraic equations. We explain this in detail in \cref{sec:2}. In short, for every pair of positive integers $M$ and $N$ we get a \emph{system of polynomial equations} with \emph{parameters}. The number $N$ is the number of oscillators in \cref{eq:coupledduff}, and $M$ is the number of \emph{leading frequencies} that are taken into account. The parameters $\alpha, \beta, \gamma, \delta, J$ are real numbers.

Our paper focuses on the case where $M = 1$ and $N$ is arbitrary. In this case, the associated system of polynomial equations has $2N$ variables $u=(u_1,\ldots,u_N)$ and $v=(v_1,\ldots,v_N)$ and $2N$ equations. Explicitly, the equations are of the form
$$\mathcal F_{N}(u,v) =
\begin{bmatrix}
\,f_{1}(u,v),\,&
\,g_{1}(u,v),\,&
\cdots&
\,f_{N}(u,v),\,&
\,g_{N}(u,v)\,
\end{bmatrix} = 0
$$
with
\begin{align}\label{eq:N}
\begin{split}
f_i(u,v) &=a_{1,i} u_i \, (u_i^2+v_i^2) + a_{2,i} \, u_i + a_{3,i} \, v_i + a_{4,i} + \sum_{j\ne i} c_{j,i}v_j,\\
g_i(u,v) &= b_{1,i} \, v_i \, (u_i^2+v_i^2) + b_{2,i} \, u_i \, + b_{3,i} \, v_i \, + b_{4,i} +  \sum_{j\ne i} d_{j,i}u_j.
\end{split}
\end{align}
The parameters of $\mathcal F_N$ are the $2(N+3)N$ numbers $a_i:=(a_{1,i},\ldots, a_{4,i})$, $b_i:=(b_{1,i},\ldots, b_{4,i})$, $c_i:=(c_{j,i}\mid j\neq i)$, and $d_i:=(d_{j,i}\mid j\neq i)$ for $i=1,\ldots,N$. They are obtained from the parameters $\alpha,\beta,\gamma,\delta, J$ above. Physically meaningful parameters are real, but we also consider complex parameters. Our main result is the following theorem.
\begin{theorem}\label{thm:main}
For general parameters $(a_i,b_i,c_i,d_i \mid i= 1,\ldots, N)\in \mathbb C^{2(N+3)N}$ the number of complex solutions of the above system of equations is
$$\#\{(u,v)\in\mathbb C^{2N} \mid \mathcal F_N(u,v)=0\} \, = \,  5^N.$$
\end{theorem}
The proof of  \cref{thm:main} is  given in \cref{sec:5}. The word ``general'' in this theorem means the following: the subset of parameters for which the statement of \cref{thm:main} does not hold, is contained in the vanishing set of a system of polynomials, where now the parameters are the variables. Let us denote this \emph{algebraic variety} by $\Delta \subset \C^{2(N+3)N}$. We call it the \emph{discriminant} associated to~$\mathcal F_N$. The discriminant $\Delta$ is a strict subvariety of the space of parameters. Therefore, the theorem holds for almost all choices of parameters $(a_i,b_i,c_i,d_i \mid i= 1,\ldots, N)\in \mathbb C^{2(N+3)N}$. Since the space of real parameters $\mathbb R^{2(N+3)N}$ is not contained in any subvariety of~$\mathbb C^{2(N+3)N}$, we must have $\mathbb R^{2(N+3)N}\not\subset \Delta$.
Therefore, \cref{thm:main} also holds for almost all real parameters.

We point out that Theorem \ref{thm:main} implies that the number of isolated complex solutions to~\eqref{eq:N} cannot exceed $5^N$. Although the theorem counts complex solutions, it follows from \cref{thm:mainN=1} that there are real parameters, for which we obtain exactly $5^N$ \emph{real} solutions. Hence our bound is optimal, also when counting real solutions.

The discriminant $\Delta$ contains the set of parameters at which the number of real solutions of $\mathcal F(u,v)=0$ is not locally constant (i.e., the number of real solutions changes when moving across these parameter values).
Computing equations for $\Delta$ is hard in general. We give an explicit description in the case $N=1$, i.e.~for a single oscillator, in \cref{sec:discriminant}. We also match this with previous results from the physics literature~\cite{kovsata2022harmonicbalance}.

Even though real coefficients and real solutions are most meaningful in our physics application, passing to the complex numbers has some advantages. Ultimately, the goal in physics is to \emph{compute} the zeros of $\mathcal F_N$ for a specific set of (real) parameters. We can use methods from computational algebraic geometry to do this. In \cref{sec:6} we describe a \emph{homotopy continuation} algorithm for computing the solutions of $\mathcal F_N(u,v)=0$. Here, passing to the complex numbers is crucial. The discriminant $\Delta$ has complex codimension at least 1 in $\mathbb C^{2(N+3)N}$, so it has real codimension at least two. Therefore, $\mathbb C^{2(N+3)N}\setminus \Delta$ is path connected, while the complement of the discriminant in $\mathbb R^{2(N+3)N}$ is disconnected. Connectedness is crucial for homotopy continuation.

Our proof of Theorem \ref{thm:main} uses the theory of \emph{Khovanskii bases} \cite{kaveh2012newton}. Root counts from Khovanskii bases can be turned into homotopy algorithms via a \emph{toric degeneration} \cite{burr2020numerical}. We emphasize that the homotopy we present in Section \ref{sec:6} does \emph{not} come from such a degeneration. It exploits the specific structure of our equations \eqref{eq:N}, which makes it particularly efficient.

The article is outlined as follows. In \cref{sec:2} we derive polynomial equations for Duffing oscillators. In \cref{sec:3} we prove \cref{thm:main} in the case $N=1$. In \cref{sec:4} we recall the concept of Khovanskii bases, which we then use in \cref{sec:5} to prove \cref{thm:main}. In \cref{sec:6} we apply our results to construct a homotopy continuation algorithm.

\begin{acknowledgement}
We thank Jan Ko{\v{s}}ata, Javier del Pino, Toni Heugel, and Oded Zilberberg for many discussion which eventually led to this research project.
\end{acknowledgement}

\section{Polynomial equations for periodic steady states} \label{sec:2}
We explain how to get to the system of polynomial equations $\mathcal F_N$ for periodic steady states.
The main results of our paper are in the case when we have $M = 1$ leading frequencies and~$N\geq 1$ oscillators. In this section, however, we assume that $M\geq 1$. This gives a more complete picture, and we anticipate follow-up works considering this general case.

\subsection{Single oscillator ($N=1$) with $M$ leading frequencies}
Let us first discuss the case $N=1$, where we only have a single oscillator and no coupling.

After a change of coordinates in \Cref{eq:duffing} we can assume that a potential periodic solution $X(t)$ is oscillating around $0$. We make the ansatz
\begin{equation}
X(t) \, = \, \sum_{k=1}^\infty \, u^{(k)} \, \sin(k \cdot t) + v^{(k)} \, \cos(k \cdot t).
\end{equation}
Assuming moreover that higher frequencies are negligible, we truncate this sum at $k = M$ and write $X^{(M)}(t)$ for the corresponding approximation. Plugging this into the left hand side of the differential equation in \eqref{eq:duffing} we obtain an equation of the form
\begin{equation}\label{F_eq}
F^{(M)}(u,v,t)=0,
\end{equation}
where $u = (u^{(1)}, \ldots, u^{(M)})$ and $v = (v^{(1)}, \ldots, v^{(M)})$. We obtain approximate values of $u^{(k)}, v^{(k)}$ for $k = 1, \ldots, M$ by requiring the $M$ leading frequencies in $F^{(M)}(u,v,t)$ to vanish.
Then, Fourier expansion allows to write
\begin{equation} \label{eq:FM}
F^{(M)}(u,v,t) \, = \, \sum_{k=1}^\infty \,  f^{(k)}(u,v) \cdot \sin(k \cdot t) + g^{(k)}(u,v) \cdot \cos(k\cdot t).
\end{equation}
The functions $f^{(k)}(u,v), g^{(k)}(u,v)$ are the Fourier coefficients given by
\[ f^{(k)}= \frac{1}{\pi} \int_{-\pi}^\pi F^{(M)}(u,v,t) \sin(k \cdot t) dt, \quad g^{(k)}= \frac{1}{\pi} \int_{-\pi}^\pi F^{(M)}(u,v,t) \cos(k \cdot t) dt. \]
Since \eqref{eq:duffing} is a polynomial equation in $(X,\dot X, \ddot X)$, we obtain from \eqref{F_eq} a system of $2M$ polynomial equations in $2M$ variables $u = (u^{(1)}, \ldots, u^{(M)})$ and $v = (v^{(1)}, \ldots, v^{(M)})$:
\begin{equation} \label{eq:sysduff}
\begin{bmatrix}
\ f^{(1)}(u,v)\ \\[0.2em]
\ g^{(1)}(u,v)\ \\[0.2em]
\vdots\\[0.2em]
\ f^{(M)}(u,v)\ \\[0.2em]
\ g^{(M)}(u,v)\
\end{bmatrix} = 0.
\end{equation}

\begin{remark}We point out that $f^{(k)}, g^{(k)}$ may be nonzero functions in $u, v$ also for $k > M$, but we only impose the first $M$ terms in \eqref{eq:FM} to vanish. This way, the resulting system of equations~\eqref{eq:sysduff} has as many equations as unknowns, and is solvable in general.
\end{remark}

Next, we work out the single oscillator polynomial system from \cref{eq:sysduff} in the case $M=1$.

\begin{example}[Polynomial system for a single oscillator and $M = 1$ leading frequencies] \label{ex:Neqto1}
We consider the first non-trivial case in which $X(t)$ is approximated by a single sine and cosine function: $M=1$. We simplify notation by setting $u^{(1)}= u$ and $v^{(1)}= v$. Plugging $X^{(1)}(t)$ into the Duffing equation \eqref{eq:duffing} gives%
\begin{align*}
F^{(1)}(u,v,t) = &-\gamma \cos(\omega t)+ \beta \cos(t)^3 v (v^2-3 u^2) -\beta \sin(t) \cos(t)^2 u  (u^2-3v^2) \\
&+\cos(t)(3\beta u^2v+\delta u +v(\alpha-1))+\sin(t) (\beta u^3+(\alpha-1)u-\delta v).
\end{align*}
The coefficients of the polynomials $f^{(1)}, g^{(1)}$ are obtained by multiplying $F^{(1)}(u,v,t)$ with $\sin(t)$ and $\cos(t)$, respectively, and integrating out $t$. For instance, the coefficient standing with $u^3$ in~$f^{(1)}$ is computed as
\[ \frac{1}{\pi} \int_{-\pi}^\pi -\beta \, \big(\sin(t)^2 \cos(t)^2 - \sin(t)^2\big) \, \mathrm d t \, = \, \frac{3}{4} \beta. \]
This way one computes that \eqref{eq:sysduff} is given by $f(u,v) = g(u,v) = 0$, where
\begin{align*}
f(u,v) &= \frac{3}{4} \beta \,  u \, (u^2+ v^{2})+(\alpha-1)\,  u-\delta  \, v , \\
g(u,v) &= \frac{8 \gamma  \sin\! \left(\pi  \omega \right) \omega +3 \pi \left(\omega +1\right) \left(\omega -1\right) \left(\beta \, v \, (  v^{2} +u^2) +(\alpha-1) \, \frac{4 \alpha \, v }{3}+\frac{4 \delta  \, u}{3}\right) }{\left(4 \omega^{2}-4\right) \pi}.
\end{align*}
Note that the polynomials $f^{(k)}, g^{(k)}$ are not all zero for $k > 1$. In fact, we compute
\[g^{(k)} = \frac{(-1)^{k-1} 2 \gamma \sin(\pi \omega)\omega}{\pi(  \omega^2 - k^2)}, \quad k \geq 4. \]
Thus, after renaming coordinates we have
\begin{align}\label{eq:N=1}
\begin{split}
f(u,v) &=a_{1} u \, (u^2+v^2) + a_{2} \, u + a_{3} \, v + a_{4},\\
g(u,v) &= b_{1} \, v \, (u^2+v^2) + b_{2} \, u \, + b_{3} \, v \, + b_{4},
\end{split}
\end{align}
which is \Cref{eq:N} for $N=1$.
\end{example}

\subsection{Polynomial equations for $N$ coupled oscillators}
The approach from the previous subsection is extended to the coupled Duffing equations given by the system of differential equations in \Cref{eq:coupledduff}.
The periodic ansatz is now
$$
X_i(t) \, = \,  \sum_{k=1}^\infty \, u_i^{(k)} \, \sin(k \cdot t) + v_i^{(k)} \, \cos(k \cdot t), \quad i = 1, \ldots, N.
$$
Truncating after $M$ terms and plugging this into \eqref{eq:coupledduff} gives $F_i^{(M)}(u,v,t), i = 1, \ldots, N$, with Fourier coefficients $f_i^{(k)}, g_i^{(k)}$. The notation $u, v$ now comprises all $u_i^{(k)}, v_i^{(k)}$. To find $u,v$, one solves a system of $2MN$ equations in $2MN$ unknowns. In matrix form, this is given by
\begin{equation} \label{eq:coupledduffeqns}
\begin{bmatrix}
\,f_1^{(1)}(u,v)\ & \cdots & \ f_N^{(1)}(u,v)\ \\[0.25em]
\ g^{(1)}(u,v)\ & \cdots & \ g_N^{(1)}(u,v)\ \\[0.25em]
&\vdots\\[0.2em]
\ f_1^{(M)}(u,v)\ & \cdots & \ f_N^{(M)}(u,v)\ \\[0.25em]
\ g_1^{(M)}(u,v)\ & \cdots & \ g_N^{(M)}(u,v)\
\end{bmatrix} = 0,
\end{equation}
In particular, \Cref{eq:sysduff} is the special case of \Cref{eq:coupledduffeqns} for $N=1$, and \Cref{eq:N} is the special case $M=1$.

\section{The case of a single oscillator} \label{sec:3}

In this section we prove \Cref{thm:main} when $N=1$. In this case we have the two degree three polynomials in two variables from \eqref{eq:N=1}:
\begin{align*}
\mathcal F_{1}(u,v) = \begin{bmatrix}\ f(u,v)\ \\ \ g(u,v)\ \end{bmatrix}
= \begin{bmatrix}
\ a_1 u \, (u^2+v^2) + a_2 \, u + a_3 \, v + a_4\ \\
\ b_1 \, v \, (u^2+v^2) + b_2 \, u \, + b_3 \, v \, + b_4\
\end{bmatrix}.
\end{align*}
Here $u,v$ are variables and $a=(a_1,\ldots,a_4), b=(b_1,\ldots,b_4)$ are $2(N+3)N=2\cdot 4=8$ parameters. The main result of this section is the following theorem, which is a special case of \Cref{thm:main}.
\begin{theorem}\label{thm:mainN=1}
For general parameters $(a,b)\in\mathbb C^8$ we have
$$\#\{(u,v)\in\mathbb C^2 \mid \mathcal F_{1}(u,v)=0\} = 5.$$
Moreover, there are choices $(a,b)\in\mathbb R^8$ for which all five solutions are real.
\end{theorem}
Furthermore, we provide an explicit description of the discriminant for $\mathcal F_1$ in \Cref{sec:discriminant} below.  Our description, among others, lets us recover the experimental pictures made in \cite{kovsata2022harmonicbalance}.



Let us now move towards the proof of \cref{thm:mainN=1}.
The theorem considers a system of equations defined on $\C^2$. It is standard practice in algebraic geometry to homogenize $f$ and $g$ to obtain a system of two homogeneous degree three equations on the projective plane $\p^2$:
\begin{align}\label{eq:N=1hom}
\begin{split}
f_{z}&=a_1 u \, (u^2+v^2) + a_2 \, uz^2 + a_3 \, vz^2 + a_4z^3 = 0,\\
g_{z}&=b_1 \, v \, (u^2+v^2) + b_2 \, uz^2 \, + b_3 \, vz^2 \, + b_4z^3 = 0.
\end{split}
\end{align}
where $u,v,z$ are the new homogeneous variables. The \emph{affine plane} $\C^2$ is contained in the space~$\p^2$: it consists of all points with $z \neq 0$. That is,~$\p^2$ is a disjoint union of $\C^2$ and the projective line $\p^1$ corresponding to $z=0$ \cite[Sections 2.2 and 2.3]{michalek2021invitation}. The advantage of passing to the projective space is that we may now apply B\'ezout's theorem. This is a classical result which tells us that, if \eqref{eq:N=1hom} has finitely many solutions in $\p^2$,
then their number (counted with multiplicity) is the product of the degrees of the polynomials, i.e.~nine. It is important to note that, to know the number of solutions to the affine system \eqref{eq:N=1}, we have to subtract from $9$ the number of solutions of \eqref{eq:N=1hom} on $\p^1 = \{z = 0\}$, counted with \emph{multiplicities}.

There are several ways to define the multiplicity of an isolated solution.
One definition uses local rings; see, e.g., \cite[Chapter 4, Definition 2.1]{cox2006using}. Intuitively speaking, the multiplicity of a solution~$(u^*:v^*:z^*) \in \p^2$ of $f_z = g_z = 0$ is the number of isolated solutions of a slightly perturbed system $f_z = \varepsilon_1 \cdot z^3, g_z = \varepsilon_2 \cdot z^3$ in a small neighbourhood of $(u^*:v^*:z^*)$. Characterizing when a solution has multiplicity $>1$ can be done via the \emph{Jacobian determinant}. If $u^* \neq 0$, we may assume $u^* = 1$ and $(1:v^*:z^*)$ has multiplicity strictly greater than one if and only if the Jacobian matrix
\[
\begin{bmatrix}
\ \frac{\partial f_z(1,v,z)}{\partial v}\ &
\ \frac{\partial f_z(1,v,z)}{\partial z}\ \\[0.5em]
\ \frac{\partial g_z(1,v,z)}{\partial v}\
&
\ \frac{\partial g_z(1,v,z)}{\partial z}\
\end{bmatrix}\hspace{-0.25em}\Biggm|_{(v,z) = (v^*,z^*)}
\]
has rank $< 2$, or equivalently, if its determinant vanishes.

The following lemma, known in various levels of generality, will be very important in our analysis. It says that the number of isolated solutions may only go down under specialization of parameters.
\begin{lemma} \label{lem:LSC}
For $i=1,\dots, n$ let $F_i\in\C[x_1,\dots,x_n,y_1,\dots,y_m]$ (the $x_i$ play the role of variables while the $y_i$ are parameters).
The function $\Phi: \C^m\rightarrow\Z$ that to $p\in \C^m$ associates the number of isolated solutions of the system $F_1(x;p)=\dots=F_n(x;p)=0$ in $\C^n$ is lower semicontinuous, meaning that the set $\{p \in \C^m \mid  \Phi(b) \leq k \}$ is Zariski closed\footnote{A subset of $\mathbb C^m$ is called Zariski closed, if it is the vanishing set of a system of polynomials.} in $\C^m$ for all $k$.
\end{lemma}
\begin{proof}
This is a direct consequence of \cite[Theorem 7.1.4]{sommese2005numerical}.
\end{proof}

\begin{remark}
We note that the assumption that we have as many polynomials as variables $x_i$ is essential. Otherwise, it is possible that the number of isolated solutions goes up in the limit. The following example is interesting: let $m=n=2$ and consider
\begin{align*}
&f_1=x_1(y_2-x_1y_1), \; f_2=x_1x_2, \; f_3=x_1(x_1^2-y_1-1),\\
&f_4=(x_2-1)(y_2-x_1y_1), \; f_5=(x_2-1)x_2, \;  f_6=(x_2-1)(x_1^2-y_1-1).
\end{align*}
For any $y_1,y_2$ this system has a solution, namely $(x_1, x_2) = (0,1)$ and for general $y_1,y_2$ this is the only solution. However, when $y_2^2-y_1^3-y_1^2=0$ there is always one additional solution (for which $x_2=0$) and when $y_1=y_2=0$ there are three isolated solutions:~$(x_1,x_2) = (1,0), (-1,0), (0,1)$. For the experts, we point out that the incidence variety defined by $f_1, \ldots, f_5$ in $\C^2 \times \C^2$ is smooth, and thus in particular Cohen-Macaulay. All fibers are finite and yet the conclusion of Lemma \ref{lem:LSC} does not hold. The reason is the failure of miracle flatness: one of the components of the incidence variety projects to a singular curve.
\end{remark}

Let us now prove \cref{thm:mainN=1}.
\begin{proof}[Proof of \cref{thm:mainN=1}]
In \cite{kovsata2022harmonicbalance}, by explicit computation, parameters for which one has five real isolated solutions were found. Hence, five is indeed a lower bound.
For the sake of completness we provide also the following direct argument.
In \eqref{eq:N=1} we set $a_1=b_1=a_3=b_2=1$ and the other parameters to zero. This yields
$ur^2 + v = vr^2 + u = 0,$
with $r^2 = (u^2+v^2)$. If $u=0$, we also have $v=0$. Conversely, if $v=0$, we also have $u=0$. If $u,v\neq 0$, we plug $v=-ur^2$ into the second equation to get
$-u(u^2+v^2)^2 + u = 0$. Since $u\neq 0$, this shows $(u^2+v^2)^2 = 1$, which has 4 solutions. In total we have $1+4=5$ isolated solutions in $\mathbb C^2$.

Now, if we put $z=0$ in  \eqref{eq:N=1hom}, we have two solutions of $f_z=g_z=0$, namely $u=iv$ and $u=-iv$. We will argue below that each of these solutions has multiplicity at least two, which leaves us with at most $5 = 9 - 2 \cdot 2$ solutions of \eqref{eq:N=1}.  As the number of isolated solutions for general parameters upper-bounds the number of isolated solutions for any set of parameters (Lemma \ref{lem:LSC}), this allows us to conclude that \eqref{eq:N=1} indeed has $5$ solutions for general parameters.

Let us now prove that the two projective solutions with $z=0$ have multiplicity at least two.
Recall that for $z=0$ we have the two solutions $u=iv$ and $u=-iv$. Using affine coordinates $u=1$ we obtain two affine polynomials $f_z(1,v,z),  g_z(1,v,z)$ in the ring $\C[v,z]$. The Jacobian matrix at the point $v^*=i, z^*=0$ is:
\[\begin{bmatrix}
2ia_1&0\\
-2b_1&0
\end{bmatrix},\]
which has zero determinant. This implies that the multiplicity of this solution is at least two. For the other solution, the reasoning remains the same, which concludes the proof.
\end{proof}


\begin{remark}The main difficulty with generalizing this analysis to  $N>1$ oscillators is that, after passing to the projective space, we obtain infinitely many solutions. In principle, it is possible to count the contribution of such solutions using the theory of characteristic classes. However, this is often infeasible in practice. {For the reader who is familiar with Bernstein's root count \cite{bernshtein1975number}, we point out that the mixed area of the Newton polygons of $f$ and $g$ is $9 > 5$.}
\end{remark}

\subsection{The discriminant}\label{sec:discriminant}

In the setting of the previous subsection ($N=1$) the discriminant is the
algebraic hypersurface
$$\Delta = \mathrm{zcl}\big(\ \{(a,b)\in\mathbb C^8 \mid \exists (u,v)\in\mathbb C^2: f(u,v)=g(u,v)= (\partial_u f \partial_vg - \partial_v f \partial_ug)(u,v)= 0\}\ \big),$$
where $\mathrm{zcl}(\,\cdot\,)$ denotes \emph{Zariski-closure}\footnote{For a subset $A\subset \mathbb C^n$ let $I(A)$ be the \emph{ideal} of polynomials in $\mathbb C[x_1,\ldots,x_n]$ vanishing on $A$. The Zariski closure of $A$ is $\mathrm{zcl}(A)=\{x\in\mathbb C^n\mid p(x)=0 \text{ for all } p\in I(A)\}$.
That is, $\mathrm{zcl}(A)$ is the smallest (by inclusion) Zariski closed set containing $A$.},
Determining the coefficients $(a,b)\in\mathbb C^8$ for which the overdetermined system $f = g = \partial_u f \, \partial_v g - \partial_v f \, \partial_u g = 0$ has a solution is a \emph{resultant} computation. We warn the reader that it does \emph{not} work to use standard determinantal formulas for the resultant, and then specialize the coefficients. For instance, the \emph{Macaulay resultant} detects solutions in projective space $\p^2$ \cite[Chapter 3]{cox2006using}. Our analysis above shows that there are two singular solutions in $\p^2$ for \emph{any} choice of parameters $(a,b)$, so that specializing the Macaulay resultant to our system gives the zero polynomial. For more on discriminants and resultants, the reader may consult the standard text book \cite{gelfand2008discriminants}.

Specialized software tools can be used for computing the discriminant $\Delta$. For instance, the package \texttt{RootFinding} provided by Maple \cite{maple} has a command \texttt{DiscriminantVariety}, which returns the polynomial defining $\Delta$. Another way to do this is via elimination of variables in Macaulay2~\cite{M2}.
Eliminating $u, v$ from $f = g = \partial_u f \, \partial_v g - \partial_v f \, \partial_u g  = 0$ is the algebraic counterpart of the geometric operation of projecting the \emph{incidence variety}
\[\mathcal I = \big\{ \big( (u,v),(a,b)\big) \in \C^2 \times \C^8 \mid f = g = \partial_u f \, \partial_v g - \partial_v f \, \partial_u g  = 0 \big\} \]
to the parameter space $\C^8$ \cite[Chapter 4]{michalek2021invitation}; i.e., $\Delta = \mathrm{zcl}(\pi(\mathcal I))$, where $\pi:\mathcal I\to \mathbb C^8$ is the projection onto the $(a,b)$ variables. Such projections rely on Gr\"obner basis computations and are computationally demanding. The following efficient code to perform elimination in Macaulay2 was provided to us by Bernd Sturmfels:
\begin{verbatim}
R = QQ[u,v,a_1..a_4,b_1..b_4, MonomialOrder=>Eliminate 2];
z = 1
f = a_1*(u^2 + v^2)*u + a_2*u*z^2 + a_3*v*z^2 + a_4*z^3
g = b_1*(u^2 + v^2)*v + b_2*u*z^2 + b_3*v*z^2 + b_4*z^3
T = diff(u,f)*diff(v,g) - diff(v,f)*diff(u,g);
F = {f, g, T};
IM = ideal F;
time G = selectInSubring(1, gens gb(IM, DegreeLimit=>24));
dis = first first entries G;
degree(dis), # terms dis
\end{verbatim}
These computations lead to the following result.
\begin{theorem}
In the case of a single oscillator ($N=1$) the discriminant $\Delta$ is the zero set of a degree 18 polynomial with with 578 terms.
\end{theorem}
We note that the discriminant subdivides real parameter space $\R^8$ into chambers in such a way that in each chamber the number of real solutions is constant. This allows us to reconstruct several special chamber decompositions from \cite{heugel2022ising, kovsata2022harmonicbalance}.

The following example provides two-dimensional cuts of the space of parameters showing the resultant hypersurface.
\begin{example}\label{example_disc}
We compare our results with \cite[Figure 6(a)]{kovsata2022harmonicbalance}, where the discriminant is plotted in a 2-dimensional slice of the space of parameters. Explicitly, one substitutes
\[
\begin{matrix}
a_1= \omega^2\eta^2 + 9 & &
a_2= (4 \eta \gamma - 12) \omega^2 + 3 (-2 \lambda + 4) \\
a_3 = 2((\lambda + 2)\eta - 6 \gamma) \omega - 4 \omega^3 \eta & &
a_4 = -12 F \cos(\Theta) - 4 \omega \eta F \sin(\Theta) \\
b_1 =  \omega^2 \eta^2 + 9 & &
b_2 = 2((\lambda - 2) \eta + 6 \gamma) \omega + 4 \omega^3 \eta \\
b_3= (4 \eta \gamma - 12) \omega^2 + 3 (2 \lambda + 4) & &
b_4 = -12 F \sin(\Theta) + 4 \omega \eta F \cos(\Theta)
\end{matrix} \]
together with $\eta = 1/2, \gamma = 1/100$ and $F = 0$.
We obtain that the discriminant is given by
$$p(\omega,\lambda)^3\cdot q(\omega,\lambda)^2\cdot \lambda^2\cdot (\omega^2 + 36)^{10}=0,$$
where
\begin{align*}
p(\omega,\lambda) &= 10000\omega^4 - 19999\omega^2 - 2500\lambda^2 + 10000\\
q(\omega,\lambda) &= 2500\omega^6 - 4700\omega^4 - 625\omega^2\lambda^2 + 2209\omega^2 - 22500\lambda^2
\end{align*}
The zero set of $p\cdot q$ in the box $[0.98,1.04]\times [0,0.04]$ can be seen in \Cref{fig_disc}. We see four regions. On each of the four regions the number of real solutions in locally constant. Indeed, in the red region all solutions are real and in the orange region we have three real solutions. In both the grey and the purple region there is only one real solution. These two regions are separated because on the curve that separates them, two complex solutions collapse. This is not visible in~\cite[Figure 6(a)]{kovsata2022harmonicbalance}, as for a random choice of parameters the probability of choosing parameters from a curve is zero.

\begin{figure}[ht]
\begin{center}
\includegraphics[width = 0.6\textwidth]{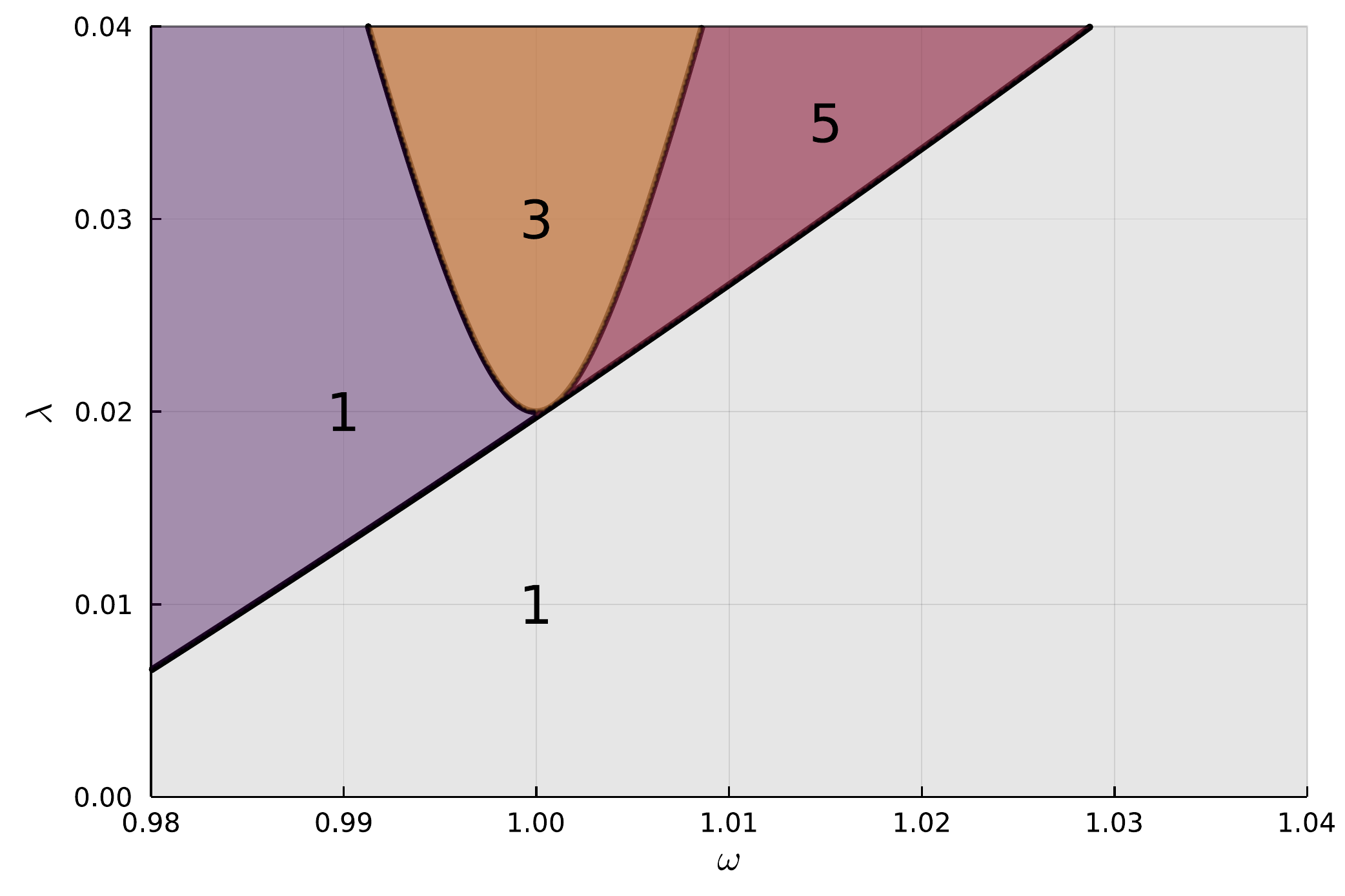}
\end{center}
\caption{\label{fig_disc} The picture shows the discriminant from \cref{example_disc} in the $(\omega,\lambda)$-plane. The discriminant is the union of the black curves. The complement of the discriminant is separated into four regions. On each region, the number of real solutions (1, 3 or 5) is locally constant.}
\end{figure}

\begin{figure}[ht]
\begin{center}
\includegraphics[width = 0.7\textwidth]{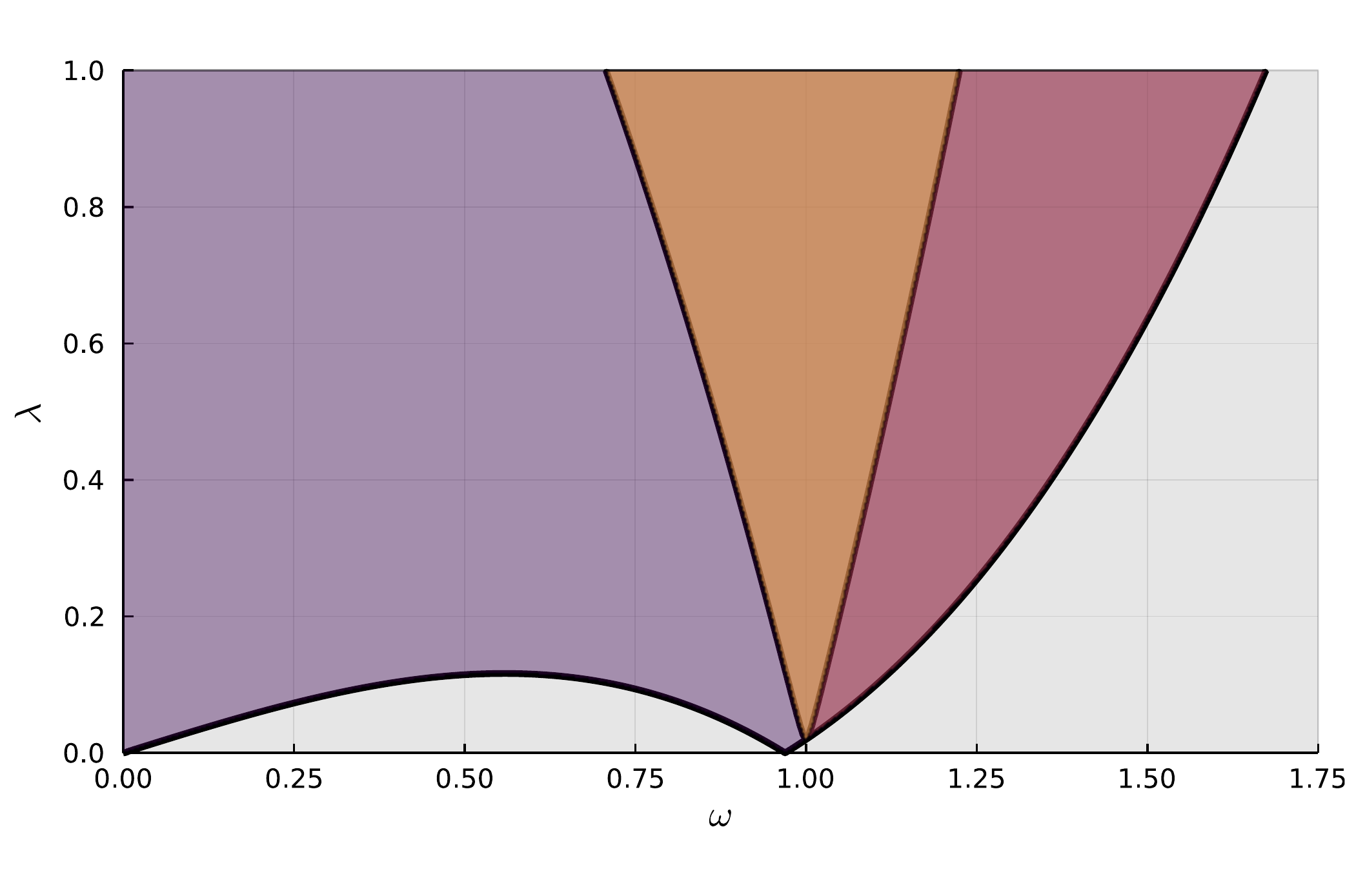}
\end{center}
\caption{\label{fig_disc2} The picture shows the discriminant from \cref{example_disc} in the $(\omega,\lambda)$-plane; the ranges of the axes are larger compared to \Cref{fig_disc}, but the color code is the same. While the curve separating the grey region from the rest appears linear in \Cref{fig_disc}, here it becomes apparent that it is not a line but a curve of higher degree.}
\end{figure}

\end{example}

\section{Preliminaries on Khovanskii bases} \label{sec:4}
In \Cref{thm:mainN=1} we proved \cref{thm:main} in the case $N=1$. The proof for arbitrary $N$ makes use of so called \emph{Khovanskii bases}. This is a concept that we now briefly explain (more details can be found in \cite{kaveh2012newton}). The proof of \cref{thm:main} comes in \cref{sec:5}.
\cref{thm:main} makes a statement on the number of solutions in $\C^n$ of a system of $n$ polynomial equations $F_1(x)=\dots=F_n(x)=0$ in $n$ variables~$x = (x_1, \ldots, x_n)$. Specifically, in \cref{thm:main} we have $n=2N$ and the $F_i$ have a certain structure.
The discussion in this section applies at the following general level. We consider $m$ (distinct) polynomials $\theta_1, \ldots, \theta_m$
such that we have $F_i = \sum_{j=1}^m c_{ij} \, \theta_j(x)$ with coefficients $c_{ij} \in \C$. There are of course many choices for such functions $h_j$. A simple criterion to see which choices fit our purpose is the following. We replace the coefficients $c_{ij}$ by arbitrary complex numbers $\tilde{c}_{ij}$ to obtain the system
\[ \sum_{j=1}^m \tilde{c}_{ij} \, \theta_j(x) = 0, \qquad i = 1, \ldots, n.\]
For almost all choices of $\tilde{c}_{ij}$, the number of solutions is a fixed positive integer. A good choice for the functions $\theta_j$ is such that $F_1(x) = \cdots = F_n(x) = 0$ has this expected number of solutions. In what follows, we explain how to compute this number.

One of the possible approaches is to change our system to a linear one, at the cost of changing the space of solutions from $\mathbb C^n$ to a more complicated one. We  consider the map
$$\phi:\C^n\dashrightarrow \p^{m-1}, \qquad x \mapsto \big[\ \theta_1(x)\  :\ \cdots\ : \ \theta_m(x)\ \big].$$
The symbol ``$\dashrightarrow$'' indicates that this map is not defined everywhere -- it is defined outside the joint zero set of the $\theta_i$. Now, for every $i$
the hypersurface $\{F_i=0\} \subset \C^n$ is mapped to a hyperplane $L_i \subset \p^{m-1}$, with linear equation $\sum_{j=1}^m c_{ij} \, y_j = 0$. Consequently, the problem of solving the polynomial system $F_1=\cdots=F_n=0$  breaks up into two subproblems:
\begin{enumerate}
\item determine the intersection of the image of $\phi$ with $r$ hyperplanes,
\item for each point in the intersection, determine the fiber of $\phi$.
\end{enumerate}
We assume from now on that the map $\phi$ is injective and that its fibers are easy to compute, which is the case for our equations~\eqref{eq:coupledduffeqns}. In this case, we may focus on step (1).
Since the $F_i$ are general linear combinations of our $m$ polynomials $\theta_j$, the number of points in the intersection from step (1) is the \emph{degree} of the closure of the image of $\phi$. We denote this by
$$X = \mathrm{zcl}({{\rm im} \, \phi}),$$
where, as before, $ \mathrm{zcl}(\,\cdot\,)$ denotes Zariski-closure\footnote{Here, we have Zariski closure in projective space $\mathbb P^{m-1}$. The definition is similar as for $\mathbb C^n$: when $I(A)$ denotes the \emph{homogeneous ideal} of polynomials vanishing on $A$, then $\mathrm{zcl}(A) = \{x\in\mathbb P^{m-1}\mid p(x)=0\text{ for all }x\in I(A)\}$.}. The degree of an algebraic variety is a central algebraic invariant. A first definition was already alluded to: the degree of a variety $X \subset \p^{m-1}$ of dimension $k$ is the number of points we obtain by intersecting $X$ with $k$ general hyperplanes.
\begin{example} \label{ex:running1}
For the system of equations \eqref{eq:N=1} in the case of a single oscillator $\phi:\C^2 \rightarrow \p^4$ is given by
\begin{equation}  \label{eq:ourphi}
\phi(u,v) = \big[\ u(u^2+v^2)\ :\ v(u^2+v^2)\ :\ u\ :\ v\ :\ 1\ \big];
\end{equation}
i.e., the polynomials in  \eqref{eq:N=1} are linear combinations of four polynomials $u(u^2+v^2), v(u^2+v^2)$, $u, v$ and a constant term.
Using homogeneous coordinates $y_1, \ldots, y_5$ on $\p^4$, we find that the closure of the image is
\[ X = \{y\in\p^4\mid \, y_2y_3-y_1y_4 =  y_3^2y_4+y_4^3-y_2y_5^2 = y_3^3+y_3y_4^2 - y_1y_5^2 = 0\}. \]
This is a surface of degree 5, as confirmed by the command \texttt{degree} in \texttt{Macaulay2}. This number can also be obtained by adding two general linear equations, which defines 5 isolated points in~$\p^4$. Note that this gives the upper bound 5 on the number of isolated solutions to \eqref{eq:N=1}.
\end{example}

\subsection{Hilbert functions}
A second definition of the degree of $X$ is more technical, but more  useful for our purpose. It uses the \emph{Hilbert polynomial} of $X$ \cite[Chapters 1 and 2]{michalek2021invitation}. In order to introduce this in our setting, let us define the graded algebra $S_\phi$ associated to the image of~$\phi$. This is the subalgebra of the polynomial ring $\C[x_1,\dots,x_n,s]$ generated by $s,s\theta_1,\dots,s\theta_m$, written $$S_\phi = \C[s, s\theta_1, \ldots, s\theta_m] \subset \C[s,x_1, \ldots, x_n].$$
The ring $\C[s,x_1, \ldots, x_n]$ is graded by setting
\begin{equation} \label{eq:grading}
\deg(s) = 1 \quad \text{and} \quad \deg(x_1) = \cdots = \deg(x_n) = 0.
\end{equation}
That is, the degree of a monomial in $\C[s, x_1, \ldots, x_n]$ is determined by its exponent in $s$. Let~$(S_\phi)_\ell$ be the vector space of degree $\ell$ elements in $S_\phi$.
The \emph{Hilbert function} is \[H_{S_\phi}:\Z_{\geq 0}\to \Z_{\geq 0}, \quad \text{given by } \quad H_{S_\phi}(\ell)=\dim_\C (S_\phi)_\ell.\] It turns out that for sufficiently large $\ell$, the function $H_{S_\phi}$ coincides with a polynomial $P_\phi$, known as the \emph{Hilbert polynomial}.
If the image of $\phi$ has dimension $k$, then the Hilbert polynomial is of degree $k$. The \emph{degree} of $X$ can be defined as $k!$ times the leading coefficient of $P_\phi$. The equivalence of this definition with the previous, geometric definition can be found in \cite[Theorem 16.9]{cutkosky2018introduction}, for instance.
This reduces the problem of computing the degree of a variety to understanding the dimensions of the vector spaces $(S_\phi)_\ell$ for large $\ell$.

\begin{example}\label{ex:running2}
We continue Example \ref{ex:running1}. The algebra $S_\phi$ is
\begin{equation} \label{eq:Sphi}
S_\phi \, = \, \C\big[\ su(u^2+v^2),\ sv(u^2+v^2),\ su,\ sv,\ s\ \big] \, \subset \,  \C[s,u,v].
\end{equation}
The degree 1 elements are linear combinations of its five generators. Hence $H_{S_\phi}(1) = 5$. Computing pairwise products of these five elements we find a basis for $(S_\phi)_2$:
\small
\begin{align*}
(S_\phi)_2 = {\rm Span}_\C \, \{\: &s^{2},\:s^{2}u,\:s^{2}v,\:s^{2}u^{3}+s^{2}u\,v^{2},\:s^{2}u^{2}v+s^{2}v^{3},\:s^{2}u^{2},\:s^{2}u\,v,\:s^{2}u^{4}+s^{2}u^{2}v^{2}, \\
&s^{2}u^{3}v+s^{2}u\,v^{3},\:s^{2}v^{2},\:s^{2}u^{2}v^{2}+s^{2}v^{4},\:s^{2}u^{6}+2\,s^{2}u^{4}v^{2}+s^{2}u^{2}v^{4},\\
&s^{2}u^{5}v+2\,s^{2}u^{3}v^{3}+s^{2}u\,v^{5},\:s^{2}u^{4}v^{2}+2\,s^{2}u^{2}v^{4}+s^{2}v^{6}\: \}.
\end{align*}
\normalsize
This shows that $H_{S_\phi}(2) = 14$. The command \texttt{hilbertFunction} in \texttt{Macaulay2} gives
\begin{equation} \label{eq:HFtable}
\begin{matrix}
\ell \, : & 0 & 1 & 2 & 3 & 4 & 5 & 6 & 7 & 8 \\
H_{S_\phi}(\ell)  \,: & 1 & 5 & 14 & 28 & 47 & 71 & 100 & 134 & 173
\end{matrix}.
\end{equation}
The Hilbert polynomial equals $P_\phi(\ell) = (5/2) \ell^2 + (3/2)\ell + 1$, which coincides with $H_{S_\phi}$ for all~$\ell \in \N$. Notice $\deg P_\phi = 2$, since $X$ is a surface. The leading coefficient confirms $\deg X = 5$.
\end{example}

The Hilbert function of $S_\phi$ is well understood when all $\theta_j$'s are monomials, i.e.~$g_j = x^{\alpha_j}$. In this case, $X = \mathrm{zcl}({\rm im} \, \phi)$ is called a \emph{toric variety} and our algebra $S_\phi$ is spanned by monomials as a $\C$-vector space, so that we are reduced to counting monomials in $S_\phi$ of a given degree. A famous result by Kushnirenko states that the leading coefficient of the Hilbert polynomial $P_\phi$ is the Euclidean volume of the polytope obtained by taking the convex hull of the exponent vectors~$\alpha_j, j = 1, \ldots, m$ in $\R^n$. This number, multiplied with $n!$, is the degree of the image variety $X$. For a proof of Kushnirenko's theorem and examples, see e.g.~\cite[Section 3.4]{telen2022introduction}.

Unfortunately, in our setting $S_\phi$ is \emph{not} generated by monomials. \cref{ex:running2} shows that in the case $N=1$ the algebra $S_\phi$ has generators $u(u^2+v^2)$ and $v(u^2+v^2)$. Similarly, for $N\geq 2$ we have generators $u_i(u_i^2+v_i^2)$ and $v_i(u_i^2+v_i^2)$ for every $i=1,\ldots,N$. Nevertheless, with some more effort one can still reduce to the monomial case. The idea is to consider the \emph{initial algebra} of~$S_\phi$ with respect to some \emph{term order} on $\C[s,x_1,\dots,x_n]$.

More precisely, fix a degree compatible term order $\preccurlyeq$ on $\C[s,x_1,\dots,x_n]$. This is a term order such that $s^{k_1}x^{\alpha_1}\preccurlyeq s^{k_2}x^{\alpha_2}$ whenever we have $\deg(s^{k_1}x^{\alpha_1})\leq \deg(s^{k_2}x^{\alpha_2})$. For a polynomial $p=\sum_{(k,\alpha) \in \mathcal{A}} c_{(k,\alpha)} s^k x^\alpha\in \C[s,x_1,\dots,x_n]$ with $c_{(k,\alpha)} \neq 0$,
define its \emph{leading monomial} to be
$$\mathrm{init}_\preccurlyeq(p) =  s^\ell x^\beta,\quad \text{ where } s^{\ell} x^\beta  \preccurlyeq  s^{k}x^\alpha\text{ for all $(k,\alpha) \in {\mathcal{A}}$}.$$
Then, the \emph{initial algebra} is spanned by the initial terms of all elements of $S_\phi$:
$$(S_{\phi})_\mathrm{in} = \mathrm{Span}_\C\big\{\mathrm{init}(p) \mid p \in S_\phi\big\}.$$
This is a subalgebra of $\C[s,x_1, \ldots, x_n]$.

\begin{proposition}
The algebras $S_\phi$ and $(S_{\phi})_\mathrm{in}$ have the same Hilbert function.
\end{proposition}
\begin{proof}
Fix $\ell$ and let $\mathcal{M}$ be a basis of leading monomials for $((S_{\phi})_\mathrm{in})_\ell$. The value of the Hilbert function of $(S_\phi)_\mathrm{in}$ at $\ell$ is
$$H_{(S_\phi)_\mathrm{in}}(\ell)=|\mathcal{M}|.$$
For each $s^kx^\alpha\in \mathcal{M}$ let $h_{k,\alpha}\in (S_\phi)_\ell$ be a polynomial with leading monomial $s^kx^\alpha$. By taking linear combinations of the $h_{k,\alpha}$ we may assume that $s^kx^\alpha$ is the only monomial from $\mathcal{M}$ in the support of $h_{k,\alpha}$.
We claim that ${\mathcal B} = \{h_{k,\alpha}\mid s^kx^\alpha \in \mathcal{M}\}$ is a basis of $(S_\phi)_\ell$, which will finish the proof. The $h_{k,\alpha}$ are linearly independent, because $\mathcal M$ is a~basis. Given any $p\in (S_\phi)_\ell$ we may subtract a linear combination $\sum_{\mu \in \mathcal B} c_\mu \mu$, so that the result involves none of the monomials in the basis~$\mathcal{M}$. But then the leading monomial of the difference $p - \sum_{\mu\in \mathcal M}c_\mu \mu\in(S_\phi)_\ell$
is not contained in $\mathcal{M}$, hence~$p - \sum_{\mu\in \mathcal M}c_\mu \mu=0$.
\end{proof}

\begin{example} \label{ex:firstkhov}
Consider the term order on $\C[s,u,v]$ for which $u^{\alpha_1}v^{\alpha_2}s^{\alpha_3} \preccurlyeq u^{\beta_1}v^{\beta_2}s^{\beta_3}$ if
\[ \alpha_3-\beta_3 < 0 \quad \text{or} \quad
\alpha_3-\beta_3 = 0,\ \alpha_2-\beta_2 > 0
\quad \text{or} \quad
\alpha_3-\beta_3 =  \alpha_2-\beta_2 =  0,\ \alpha_1-\beta_1 > 0.
\]
In other words, $u^{\alpha_1}v^{\alpha_2}s^{\alpha_3} \preccurlyeq u^{\beta_1}v^{\beta_2}s^{\beta_3}$, if the first nonzero entry of $M \cdot (\alpha - \beta)$ is negative, where $\alpha = (\alpha_1, \alpha_2, \alpha_3)^\top$, $\beta = (\beta_1, \beta_2, \beta_3)^\top$ and
\[ M = \begin{bmatrix}
0 & 0 & 1 \\ 0 & -1 & 0 \\ -1 & 0 & 0
\end{bmatrix}. \]
The initial terms of the generators $s\theta_j$ of $S_\phi$ from \eqref{eq:Sphi} are $suv^2, sv^3,su,sv$ and $s$. These form a basis of the degree 1 part of the initial algebra $(S_{\phi})_\mathrm{in}$. Higher degree monomials in $(S_{\phi})_\mathrm{in}$ are obtained as the leading monomial of a linear combination of products of the $\theta_j$. This particular set of polynomials $\theta_j$ has the special property that $(S_{\phi})_\mathrm{in}$ is generated in degree 1:
\begin{equation} \label{eq:khovprop}
 \C[\text{ leading monomials of $s\theta_1, \ldots, s\theta_5$ } ] = (S_{\phi})_\mathrm{in}.
\end{equation}
We will see below that this makes $s\theta_1, \ldots, s\theta_5$ a \emph{Khovanskii basis} of $S_\phi$. Figure \ref{fig:semigroup} shows the graded pieces of $(S_{\phi})_\mathrm{in}$ of degree $1, 2, 3$ and $4$. Here a lattice point $(\alpha_1, \alpha_2, \alpha_3)$ is identified with the monomial $u^{\alpha_1}v^{\alpha_2}s^{\alpha_3}$. The number of marked points at level $\ell$ is seen in~\eqref{eq:HFtable} as $H_{S_\phi}(\ell)$.
\begin{figure}
\centering
\begin{tikzpicture}[scale = 0.85]

\begin{axis}[%
width=4in,
height=3in,
at={(0.926in,0.661in)},
scale only axis,
xmin=-0.5,
xmax=5,
tick align=outside,
ymin=-0.5,
ymax=13,
zmin=0,
zmax=5,
ztick = {10},
view={75}{23.6},
hide axis
]

  \draw[->] (axis cs:0,0,0) -- (axis cs:0,12,0) node[right] {$v$};
    \draw[->] (axis cs:0,0,0) -- (axis cs:0,0,5) node[right] {$s$};
      \draw[->] (axis cs:0,0,0) -- (axis cs:4,0,0) node[right] {$u$};

\addplot3 [color=blue,solid,thick, dashed, fill opacity = 0.1, fill = blue,forget plot]
  table[row sep=crcr]{%
0 0 1 \\
0 1 1 \\
0 3 1 \\
1 2 1 \\
1 0 1 \\
0 0 1 \\
};

\addplot3 [color=blue,draw=none, mark=*, mark options={solid, blue}]
 table[row sep=crcr] {%
1 2 1 \\
0 3 1 \\
1 0 1 \\
0 1 1 \\
0 0 1 \\
};

\addplot3 [color=blue!60!white,solid,thick, dashed, fill opacity = 0.1, fill = blue!60!white,forget plot]
  table[row sep=crcr]{%
0 0 2 \\
0 2 2 \\
0 6 2 \\
2 4 2 \\
2 0 2 \\
0 0 2 \\
};

\addplot3 [color=blue!60!white,draw=none, mark=*, mark options={solid, blue!60!white}]
 table[row sep=crcr] {%
2 4 2 \\
1 5 2 \\
0 6 2 \\
2 2 2 \\
1 3 2 \\
0 4 2 \\
1 2 2 \\
0 3 2 \\
2 0 2 \\
1 1 2 \\
0 2 2 \\
1 0 2 \\
0 1 2 \\
0 0 2 \\
};

\addplot3 [color=blue!20!white,solid,thick, dashed, fill opacity = 0.1, fill = blue!20!white,forget plot]
  table[row sep=crcr]{%
0 0 3 \\
0 3 3 \\
0 9 3 \\
3 6 3 \\
3 0 3 \\
0 0 3 \\
};

\addplot3 [color=blue!20!white,draw=none, mark=*, mark options={solid, blue!20!white}]
 table[row sep=crcr] {%
3 6 3 \\
2 7 3 \\
1 8 3 \\
0 9 3 \\
3 4 3 \\
2 5 3 \\
1 6 3 \\
0 7 3 \\
2 4 3 \\
1 5 3 \\
0 6 3 \\
3 2 3 \\
2 3 3 \\
1 4 3 \\
0 5 3 \\
2 2 3 \\
1 3 3 \\
0 4 3 \\
3 0 3 \\
2 1 3 \\
1 2 3 \\
0 3 3 \\
2 0 3 \\
1 1 3 \\
0 2 3 \\
1 0 3 \\
0 1 3 \\
0 0 3 \\
};

\addplot3 [color=blue!10!white,solid,thick, dashed, fill opacity = 0.1, fill = blue!10!white,forget plot]
  table[row sep=crcr]{%
0 0 4 \\
0 4 4 \\
0 12 4 \\
4 8 4 \\
4 0 4 \\
0 0 4 \\
};

\addplot3 [color=blue!10!white,draw=none, mark=*, mark options={solid, blue!10!white}]
 table[row sep=crcr] {%
4 8 4 \\
3 9 4 \\
2 10 4 \\
1 11 4 \\
0 12 4 \\
4 6 4 \\
3 7 4 \\
2 8 4 \\
1 9 4 \\
0 10 4 \\
3 6 4 \\
2 7 4 \\
1 8 4 \\
0 9 4 \\
4 4 4 \\
3 5 4 \\
2 6 4 \\
1 7 4 \\
0 8 4 \\
3 4 4 \\
2 5 4 \\
1 6 4 \\
0 7 4 \\
4 2 4 \\
3 3 4 \\
2 4 4 \\
1 5 4 \\
0 6 4 \\
3 2 4 \\
2 3 4 \\
1 4 4 \\
0 5 4 \\
4 0 4 \\
3 1 4 \\
2 2 4 \\
1 3 4 \\
0 4 4 \\
3 0 4 \\
2 1 4 \\
1 2 4 \\
0 3 4 \\
2 0 4 \\
1 1 4 \\
0 2 4 \\
1 0 4 \\
0 1 4 \\
0 0 4 \\
};

\draw (axis cs:1,3,1) node{\textcolor{blue}{$Q$}};

\end{axis}
\end{tikzpicture}
\caption{Monomial basis of the algebra $(S_{\phi})_\mathrm{in}$ from Example \ref{ex:firstkhov}.}
\label{fig:semigroup}
\end{figure}
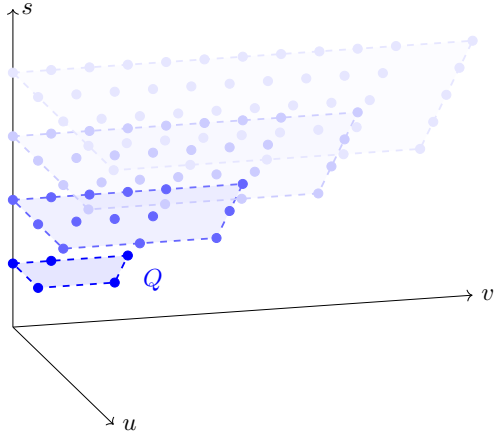
\end{example}
\subsection{Khovanskii bases}
Motivated by \cref{ex:firstkhov} we have the following definition.
\begin{definition}[Khovanskii basis]
A subset of elements of $S_\phi$ is called a \emph{Khovanskii basis} if the leading monomials of those elements generate the initial algebra $(S_{\phi})_\mathrm{in}$.
\end{definition}
\begin{remark}
Khovanskii bases are defined as more general objects, see \cite{kaveh2012newton}. In the special setting considered in this paper, they are sometimes called \emph{SAGBI bases}, or \emph{canonical subalgebra bases}.
\end{remark}
The importance of the concept of Khovanskii bases is given by the next theorem.

\begin{theorem} \label{thm:mainkhov}
Consider $\theta_1, \ldots, \theta_m \in \C[x_1, \ldots, x_n]$ and $\phi = (\theta_1, \ldots, \theta_m): \C^n \dashrightarrow \p^{m-1}$ with image closure $X = \mathrm{zcl}({{\rm im}\, \phi})$. Let $\preccurlyeq$ be a term order on  $\C[s, x_1, \ldots, x_n]$ compatible with its grading \eqref{eq:grading}, and consider the graded algebra $S_\phi = \C[s\theta_1, \ldots, s\theta_m]$. Furthermore, define
$$
Q = {\rm Conv}\big\{ \alpha_j \in \Z^n \mid (\alpha_j, 1) \text{ is the exponent of the $\preccurlyeq$-leading monomial of $s\theta_j$} \big \} \, \subset \R^n.
$$
If the polynomials $s\theta_1, \ldots, s\theta_m$ form a Khovanskii basis for $S_\phi$ and $\dim Q = n$, then
$$\deg X = n! \cdot {\rm Vol}(Q).$$
\end{theorem}
When all $\theta_j$ are monomials, the generators $s\theta_1,\dots,s\theta_m$ automatically form a Khovanskii basis of $S_{\phi} = \C[s\theta_1, \ldots, s\theta_m]$. Example \ref{ex:firstkhov} shows that this happens in other cases. However, we warn the reader that, in general, generators of the algebra may not form a Khovanskii basis.

Important for us is that there is a finite algorithm to check whether a set of algebra generators $s\theta_1, \ldots, s\theta_m$ is a Khovanskii basis for $S_\phi = \C[s\theta_1, \ldots, s\theta_m]$, see \cite[Corollary 11.5]{sturmfels1996grobner}. We will use this below to prove that the algebras encountered in coupled oscillator systems come with a finite set of generators that is a Khovanskii basis. In this case, \cref{thm:mainkhov} implies that the degree of our variety $X = \mathrm{zcl}({ {\rm im} \, \phi})$ can be computed in terms of the volume of a polytope.

Note that, once \cite[Corollary 11.5]{sturmfels1996grobner} has been applied to show \eqref{eq:khovprop}, Theorem \ref{thm:mainkhov} provides a new proof of Theorem \ref{thm:mainN=1}. Indeed, the degree $\deg(X) = 2! \cdot {\rm Area}(Q) = 5$ is an upper bound for the general number of solutions to \eqref{eq:N=1}. The lower bound is provided by the instance with 5 distinct solutions that we construct in the proof of \cref{thm:mainN=1} in \Cref{sec:3}.
Similarly, \cref{thm:mainkhov} will help us prove the number of solutions for arbitrary $N$.

\begin{remark}[Newton-Okounkov Bodies]
Theorem \ref{thm:mainkhov} can be generalized to the case where our generators $s\theta_j$ do not form a Khovanskii basis of $S_\phi$. For each $\ell$ we obtain the polytope $Q_\ell$ as the convex hull of the exponents of all monomials in $((S_{\phi})_\mathrm{in})_\ell$.
We consider the sequence of polytopes $\tilde{Q}_\ell = (1/\ell) Q_\ell$, which satisfies $\tilde{Q}_\ell \subset \tilde{Q}_{\ell+1}$. The limit $Q_\infty =\lim_{\ell \rightarrow\infty}\tilde{Q}_\ell=\overline{\bigcup_\ell \tilde{Q}_\ell}$  is a convex body, known as
the \emph{Newton-Okounkov-Body} of $S_\phi$. The key result of \cite{kaveh2012newton} is that the normalized volume of $Q_\infty$ is the degree of
$X = \mathrm{zcl}({\rm im}\, \phi)$,
assuming $\phi$ is injective. Note that, in the case of Theorem \ref{thm:mainkhov}, we have $Q = Q_1 = \tilde{Q}_\ell = Q_\infty$.
\end{remark}

\section{The case of $N$ oscillators} \label{sec:5}
In this section, we apply the tools from Section \ref{sec:4} to compute the general number of solutions to the equations~ \eqref{eq:N}. Recall that we have $2N$ equations which come in $N$ pairs:
\begin{align}\label{eq:N-sec5}
\begin{split}
f_i(u,v) &=a_{1,i} u_i \, (u_i^2+v_i^2) + a_{2,i} \, u_i + a_{3,i} \, v_i + a_{4,i} + \sum_{j\ne i} c_{j,i}v_j = 0 ,\\
g_i(u,v) &= b_{1,i} \, v_i \, (u_i^2+v_i^2) + b_{2,i} \, u_i \, + b_{3,i} \, v_i \, + b_{4,i} +  \sum_{j\ne i} d_{j,i}u_j = 0.
\end{split}
\end{align}
Here $u_i,v_i$ are variables, and $a_{j,i}, b_{j,i}, c_{j,i}$ and $d_{j,i}$ are parameters. For brevity, we will denote the number of isolated complex solutions for general parameter values by $\delta_N$. We first observe a lower bound by choosing a special set of parameters.

\begin{lemma} \label{lem:lowerbound}
The general number of solutions $\delta_N$ is at least $5^N$.
\end{lemma}

\begin{proof}
We consider systems of the form \eqref{eq:N-sec5} with parameters $c_{j,i}, d_{j,i}$ specialized to $0$ and $a_{j,i}, b_{j,i}$ general. Such systems correspond to decoupled oscillators and are given by $N$ independent systems of type~\eqref{eq:N=1}. There are $5^N$ solutions by Theorem \ref{thm:mainN=1}. The lemma follows from Lemma \ref{lem:LSC}: the number of isolated solutions with specialized parameters is at most $\delta_N$.
\end{proof}

To prove Theorem~\ref{thm:main}, it remains to show the upper bound $\delta_N \leq 5^N$. For this, we proceed as in Section \ref{sec:4}. Let $u_1, \ldots, u_N, v_1, \ldots, v_N$ be coordinates on $\C^{2N}$ and consider the polynomials
\begin{equation} \label{eq:paramY}
\theta_0 = 1,\;
\theta_{i,1} = u_i,\;
\theta_{i,2} = v_i,\;
\theta_{i,3} = (u_i^2 + v_i^2)u_i,\;
\theta_{i,4} = (u_i^2 + v_i^2)v_i, \quad i = 1, \ldots, N.
\end{equation}
These define an injective parametrization map $\phi: \C^{2N} \dashrightarrow \p^{4N}$ whose image closure we denote
$$Y_N := \mathrm{zcl}({\rm im} \, \phi) \subset \p^{4N}.$$
We have that $Y_N$ is a $2N$-dimensional projective variety, and $f_i = g_i = 0, i = 1, \ldots, N$, push forward to $2N$ linear equations on $Y_N$. Here is a consequence.

\begin{lemma} \label{lem:upperbound}
The general number of solutions $\delta_N$ is at most $\deg(Y_N)$.
\end{lemma}
We are left with showing that $\deg(Y_N) = 5^N$. For this, we seek to apply Theorem \ref{thm:mainkhov}. Throughout the section, we use the following monomial order on $\C[s,u_1, \ldots, u_N, v_1, \ldots, v_N]$. In analogy with Example \ref{ex:firstkhov}, $u^{\alpha_1}v^{\alpha_2}s^{\alpha_3} \preccurlyeq u^{\beta_1}v^{\beta_2}s^{\beta_3}$,
if the first nonzero entry of $M \cdot (\alpha - \beta)$ is negative, where $\alpha = (\alpha_1, \alpha_2, \alpha_3)^\top
\neq \beta = (\beta_1, \beta_2, \beta_3)^\top$ and $M$ is the $(2N+1) \times (2N+1)$ matrix
\begin{equation}  \label{eq:monorder}
M = \begin{bmatrix}
 0 & \cdots & 0 & 1 \\
 0 & \cdots & -1 & 0 \\
 0 & \iddots & 0 & 0 \\
 -1 & \cdots & 0 & 0
\end{bmatrix}.
\end{equation}
Here $\alpha_1,\alpha_2,\beta_1,\beta_2$ are $N$-vectors and $u^{\alpha_i}$ is the standard multi-index notation.

\subsection{The Oscillator Polytope}
We proceed by studying the polytope which plays the role of $Q$ in Theorem \ref{thm:mainkhov}. We call it the \emph{Oscillator Polytope}\footnote{The definition of the polytope does not just depend on the equations, but also on the choice of a term order, hence is not canonical. We still find the name appropriate, because our choice of term order gives a particularly simple polytope. Moreover, the name is supposed to underline the deep connection between counting zeros and convex geometry that is exploited here.}.

The initial terms of the polynomials $s h_0, sh_{i,j}$ from \eqref{eq:paramY} that correspond to the monomial order in~\eqref{eq:monorder} are $ s \cdot \{1,u_1,v_1,u_1v_1^2,v_1^3,\dots ,u_N,v_N,u_Nv_N^2,v_N^3\}$. The oscillator polytope is the following $(2N)$-dimensional polytope in $\R^{2N} = (\R^2)^N$:
\begin{equation} \label{eq:PN}
Q_N = \Conv \Big ( \{ 0 \} \cup \bigcup_{i = 1}^N \left \{ e_{i,1} \, , \, 3 e_{i,2} \, ,  \, e_{i,1} + 2 e_{i,2} \right \} \Big ) \subset (\R^2)^N,
\end{equation}
where $e_{i,1}, e_{i,2}$ are the standard basis vectors on the $i$-th copy of $\R^2$. Our goal in this subsection is to prove the following Lemma.
\begin{lemma} \label{lem:volQN}
The volume of the oscillator polytope $Q_N$ satisfies
$(2N)! \cdot {\rm Vol}(Q_N) = 5^N$.
\end{lemma}
For varying $N$, the polytopes~$Q_N$ can be described using \emph{subdirect sums}.
\begin{definition}
Let $P\subset \R^n, Q\subset \R^m$  be two polytopes. Their subdirect sum $P\oplus Q\subset \R^{n+m}$ is given by
$
P\oplus Q =\Conv \left( \{(p,0) \mid p\in P\} \cup \{(0,q) \mid q\in Q\}  \right).
$
\end{definition}
The oscillator polytope is obtained by taking $N-1$ subdirect sums:
$$Q_N= Q_1\oplus \ldots \oplus Q_1.$$
This structure of $Q_N$ can be used to compute its volume. The following lemma uses the \emph{normalized volume}, which for a polytope $P$ in $\R^n$ is defined as ${\rm vol}(P) = n! \cdot {\rm Vol}(P)$, where ${\rm Vol}(P)$ is the Euclidean volume.

\begin{lemma}\label{lem:subdirectvol}
Let $P\subset \R^n, Q\subset \R^m$  be two polytopes containing the origins of $\R^n$ and $\R^m$ respectively. Assume that $P$ and $Q$ admit triangulations $P=\bigcup_{k \in K} \Delta_k$ and $Q = \bigcup_{\ell \in L} \Delta'_\ell$ such that $0\in \Delta_k$, $0\in \Delta'_\ell$ for any $k\in K$, $\ell\in L$.
Then the normalized volume of the subdirect sum~$P\oplus Q$ is the product of normalized volumes of $P$ and $Q$:
\[
\vol(P\oplus Q)=\vol(P)\cdot\vol(Q).
\]
\end{lemma}
\begin{proof}
We have the triangulation
$
P\oplus Q = \bigcup_{k \in K} \bigcup_{\ell \in L} \Conv( \Delta_k \cup \Delta'_\ell).
$
Therefore, the normalized volume of $P\oplus Q $ can be computed as
$$\vol(P\oplus Q) = \sum_{k\in K,\ell\in L}\vol(\Conv( \Delta_k \cup \Delta'_\ell)).$$ Since $0\in \Delta_k$ and $0\in \Delta'_\ell$ for any $k\in K$, $\ell\in L$ we have $\vol(\Conv( \Delta_k \cup \Delta'_\ell))=\vol(\Delta_k)\cdot\vol(\Delta'_\ell)$ and thus we conclude that $\vol(P\oplus Q) =  \vol(P) \cdot \vol(Q)$.
\end{proof}

An important case in which the assumptions of Lemma~\ref{lem:subdirectvol} are satisfied is when the polytopes $P$ and $Q$ have  a vertex at the origin. In that case, one can produce desired triangulations in the following way. Let $\mathcal{F}$ be the set of faces of $P$ which do not contain the origin. Consider a triangulation $\bigcup_{F\in\mathcal{F}}F = \bigcup_{k \in K} \Delta_k$, such that each simplex $\Delta_k$ belongs to a unique face~$F\in \mathcal{F}$. A desired triangulation of $P$ can be constructed as
$
P=\bigcup_{k \in K}\Conv(\{0\}
\cup \Delta_k).
$
Consequently, we obtain the following corollary of Lemma~\ref{lem:subdirectvol}.

\begin{corollary}\label{cor:subdirectvol}
Let $P\subset \R^n, Q\subset \R^m$  be two polytopes having a vertex at the origins of $\R^n$ and $\R^m$ respectively. Then the normalized volume of the subdirect sum $P\oplus Q$ is the product of normalized volumes of $P$ and $Q$:
\[
\vol(P\oplus Q)=\vol(P)\cdot\vol(Q).
\]
\end{corollary}

\begin{remark} It is crucial that we work with \emph{normalized} volume in Lemma~\ref{lem:subdirectvol}. For example, let $\Delta_n =\Conv\{0,e_1,\ldots e_n\}$ be a basis simplex. The Eucledian volume of $\Delta_n$ equals $1/n!$ and hence the normalized volume is $\vol(\Delta_n)=1$. Evidently, we have $\Delta_n\oplus \Delta_k = \Delta_{n+k}$, so Euclidean volumes are not multiplicative but the normalized volumes are. \end{remark}

We apply \cref{cor:subdirectvol} to compute the volume of the oscillator polytope.

\begin{proof}[Proof of Lemma \ref{lem:volQN}]
Since $Q_N= Q_1\oplus \ldots \oplus Q_1$ and $Q_1$ has a vertex at the origin,  $\vol(Q_N)=\vol(Q_1)^N$ by \cref{cor:subdirectvol}. The lemma follows from the fact that $\vol(Q_1)=5$, which can be computed directly.
\end{proof}

\subsection{A toric ideal} To apply Theorem \ref{thm:mainkhov}, we need to show that the polynomials $s h_0, s h_{i,j}$ from \eqref{eq:paramY} form a Khovanskii basis for the subalgebra they generate.  For this, we want to apply \cite[Theorem 2.17]{Kaveh-Manon}, which requires us to study the toric variety corresponding to our leading monomials. Let $X_N \subset \p^{4N}$ be the projective toric variety obtained from the monomial map
\begin{equation} \label{eq:paramX}
(u_1,v_1,\dots,u_N,v_N)\mapsto [1:u_1:v_1:u_1v_1^2:v_1^3:\dots :u_N:v_N:u_Nv_N^2:v_N^3].
\end{equation}
On $\p^{4N}$ we will be using the coordinate system $y_0,y_{1,1},y_{1,2},y_{1,3},y_{1,4},y_{2,1},\dots,y_{N,4}$.
Our next goal is to find binomial generators of the defining ideal $I(X_N)$. We consider the affine chart $\tilde{X}_N \subset \C^{4N}$ of $X_N$ corresponding to $\{y_0 \neq 0 \}$. This is the $2N$-dimensional affine toric variety obtained from
\[(u_1,v_1,\dots,u_N,v_N)\mapsto (u_1,v_1,u_1v_1^2,v_1^3,\dots ,u_N,v_N,u_Nv_N^2,v_N^3).\]
On $\C^{4N}$ we will use the coordinate system $y_{1,1},y_{1,2},y_{1,3},y_{1,4},y_{2,1},\dots,y_{N,4}$. To simplify the notation, when $N = 1$ we write $y_j$ instead of $y_{1,j}$. The following technical lemma is needed for Theorem \ref{thm:2elemsubsets}, which allows us to simplify the computation of $I(X_N)$ to that of $I(X_2)$.
The result uses the notion of the \emph{Graver basis} of a toric ideal. A binomial $x^\alpha - x^\beta$ in a toric ideal $J$ is called \emph{primitive} if there exists no other binomial $x^{\alpha'}-x^{\beta'} \in J$ such that~$x^{\alpha'}$ divides~$x^\alpha$ and~$x^{\beta'}$ divides~$x^\beta$. The Graver basis of $J$ is the set of all primitive binomials.
The elements of the Graver basis generate $J$. The reader can consult \cite[Chapter 4]{sturmfels1996grobner} for more information.
\begin{lemma}\label{lem:GraverDecomposition}
Let $J$ be a toric ideal and $G$ its Graver basis. For any binomial $m_1-m_2\in J$ there exist binomials $\mu_{1,i}-\mu_{2,i}\in G$ for $i=1,\dots,k$ such that $m_1=\prod_{i=1}^k \mu_{1,i}$ and $m_2= \prod_{i=1}^k \mu_{2,i}$.
\end{lemma}
\begin{proof}
The proof is by induction on $\deg m_1+\deg m_2$. If $m_1-m_2\in G$ then we can take $\mu_{1,1}=m_1$ and $\mu_{2,1}=m_2$. Otherwise, there exists $\mu_{1,1}-\mu_{2,1}\in G$ such that $\mu_{1,1}|m_1$ and $\mu_{2,1}|m_2$. Note that the binomial $\frac{m_1}{\mu_{1,1}}-\frac{m_2}{\mu_{2,1}}$
belongs to $J$, as it still encodes an integral relation among exponents of monomials parameterizing the toric variety. Since $\frac{m_1}{\mu_{1,1}}-\frac{m_2}{\mu_{2,1}}$ has lower degree than~$m_1-m_2$, we can conclude by induction.
\end{proof}

Recall that $\tilde{X}_N \subset \C^{4N}$ denotes the ideal of the affine chart of $X_N$ where $\{y_0 \neq 0 \}$.
\begin{lemma}\label{lemma:Graver}
The Graver basis of $I(\tilde{X}_1)$ has five elements:
\[y_1y_2^2-y_3,\ y_2^3-y_4,\ y_1y_4-y_2y_3,\ y_1^2y_2y_4-y_3^2,\ y_1^3y_4^2-y_3^3.\]
\end{lemma}
\begin{proof}
We apply the software \texttt{4ti2} \cite{4ti2} in \texttt{Macaulay2} \cite{M2}:
\begin{verbatim}
loadPackage "FourTiTwo"
A = matrix{{1,0,1,0}, {0,1,2,3}}
toricGraver(A)
\end{verbatim}
The result consists of the five polynomials in the statement of the lemma.
\end{proof}
For any subset ${\mathcal I} \subset\{1,\dots, N\}$, we let $X_{\mathcal I}$ be the toric variety obtained from projecting $X_N$ to the $\p^{4|{\mathcal I}|}$ with coordinates $\{y_0\} \cup \{ y_{i,s} \}_{i \in \mathcal I, s =1, \ldots, 4}$. The ideal $I(X_{\mathcal I})
$ is naturally viewed as a subideal of $I(X_N)$.

\begin{theorem} \label{thm:2elemsubsets}
Using the notation introduced above, the ideal $I(X_N)$ is generated by the union of the generators of $I(X_{\{i,j\}})$ for $1\leq i <j \leq N$.
\end{theorem}
\begin{proof}
We will prove a stronger statement: that the union of the Graver bases for $I(X_{\{i,j\}})$, $1 \leq i < j \leq N$ is a Graver basis for $I(X_N)$. We write $G_N$ for the Graver basis of $I(X_N)$, and $G_{\mathcal I}$ for that of $I(X_{\mathcal I})$. We have $\bigcup_{\{i,j\}} G_{\{i,j\}} \subset G_N$ by construction. To prove the other inclusion, let us fix a binomial $m_1-m_2\in G_N$.
We have $\deg(m_1) = \deg(m_2)$, and need to show that~$m_1 - m_2 \in G_{\{i,j\}}$ for some $i,j$.
Without loss of generality we may write:
\[m_1=y_0^k\prod_i m_{1,i} ,\quad m_2=\prod_{i} m_{2,i} \quad\text{ with }\quad m_{1,i}=\prod_{s} y_{i,s}^{k_{i,s}},\quad m_{2,i}=\prod_{s} y_{i,s}^{\ell_{i,s}} .\]
Let us fix $i$. If there exists some $k_{i,s} \neq 0$, or some $\ell_{i,s} \neq 0$, then $m_{1,i}- m_{2,i}$ lies in $I(\tilde{X}_{\{i\}})$.

If we have $\sum_{s}k_{i,s} = \sum_{j} \ell_{i,s}$, then $m_{1,i}- m_{2,i} \in I(X_{\{i\}}) \subset I(X_N)$ and we must have that~$m_1 - m_2 \in G_{\{i\}} \subset G_{\{i,j\}}$ for any $j \neq i$. We may therefore assume  $\sum_{s}k_{i,s} \neq \sum_{s} \ell_{i,s}$.

Fix $i$ such that $\sum_{s}k_{i,s}  < \sum_{s} \ell_{i,s}$. Such an $i$ exists, because $k \geq 0$. By Lemma \ref{lem:GraverDecomposition} and since~$m_{1,i}- m_{2,i} \in I(\tilde{X}_{\{i\}})$,
there is an element $b = \mu_1 - \mu_2$ in the Graver basis of~$I(\tilde{X}_{\{i\}})$ such that $\mu_1 | \prod_{s} y_{i,s}^{k_{i,s}} $ and $\mu_2 | \prod_{s} y_{i,s}^{\ell_{i,s}} $
and $\deg(\mu_2) \geq \deg(\mu_1)$. By Lemma \ref{lemma:Graver}, the difference between the degrees of $\mu_1,\mu_2$ must be $\deg(\mu_2) - \deg(\mu_1) \in \{0,2\}$. We now have two cases:
\begin{enumerate}
\item If $\deg(\mu_2) = \deg(\mu_1)$, then the fact that $m_1 - m_2$ is primitive ensures that we have $m_1 - m_2=\mu_1-\mu_1 \in  I(X_{\{i\}}) \subset I(X_N)$.
 As before we conclude $m_1 - m_2 \in G_{\{i\}} \subset G_{\{i,j\}}$ for any $j \neq i$.
\item If $\deg(\mu_2) - \deg(\mu_1) = 2$, there are two subcases:
\begin{enumerate}
\item if $k \geq 2$, then, as $m_1-m_2$ is primitive, we must have $m_1 - m_2 = y_0^2\mu_1 - \mu_2 \in I(X_{\{i\}})$.   
Thus, $m_1 - m_2 \in G_{\{i\}} \subset G_{\{i,j\}}$ for any $j \neq i$,
\item if $k = 0$ or $k = 1$, there is $j \neq i$ such that $\sum_{s}k_{j,s}  > \sum_{s} \ell_{j,s}$. Analogously to the choice of $b$, by Lemma \ref{lem:GraverDecomposition}, we may find $b' = \mu_1' - \mu_2'$ in the Graver basis of $I(\tilde{X}_{\{j\}})$ with $\deg(\mu_1') - \deg(\mu_2') = 2$.
In this case the fact that $m_1 - m_2$ is primitive, implies that~$m_1 - m_2 = \mu_1\mu_1' - \mu_2\mu_2' \in G_{\{i,j\}}$.
\qedhere
\end{enumerate}
\end{enumerate}
\end{proof}

One computes that the toric ideal $I(X_{\{i,j\}})$ is generated by 10 binomials
\begin{equation} \label{eq:toricidealij}
\begin{array}{lll}
 y_{i,2}y_{i,3} - y_{i,1}y_{i,4}  & & y_{j,2}y_{j,3} - y_{j,1}y_{j,4} \\[0.3em]
y_0^2y_{i,4} - y_{i,2}^3 & & y_0^2y_{j,4} - y_{j,2}^3 \\[0.3em]
 y_0^2y_{i,3} - y_{i,1}y_{i,2}^2  &&   y_0^2y_{j,3} - y_{j,1}y_{j,2}^2 \\[0.3em]
 y_0^2y_{i,3}^2 - y_{i,1}^2y_{i,2}y_{i,4} & & y_0^2y_{j,3}^2 - y_{j,1}^2y_{j,2}y_{j,4} \\[0.3em]
 y_{i,1}^2y_{i,2}y_{i,4}y_{j,3}^2 - y_{i,3}^2y_{j,1}^2y_{j,2}y_{j,4} & &  y_{i,1}^3 y_{i,4}^2 y_{j,3}^2 - y_{i,3}^3 y_{j,1}^2y_{j,2}y_{j,4}.
\end{array}
\end{equation}
Theorem \ref{thm:2elemsubsets} provides a set of $5 N(N-1)$ generators for $I(X_N)$ by considering all 2-element subsets $\{i,j\} \subset \{1, \ldots, N\}$.


\subsection{Khovanskii basis}
We are now ready to show that the polynomials $\theta_0, \theta_{i,j}$ from \eqref{eq:paramY} form a Khovanskii basis for the subalgebra they generate:
\[S_\phi  \, = \, \C[s, s\theta_{1,1}, s \theta_{1,2}, s \theta_{1,3}, s\theta_{1,4}, \ldots, s\theta_{N,1}, s \theta_{N,2}, s \theta_{N,3}, s\theta_{N,4}] \subset \C[s,u_1,\ldots, u_N, v_1, \ldots, v_N].\]
Using Theorem \ref{thm:2elemsubsets}, we will show that this claim for general $N$ follows from the statement for~$N = 2$. We write $S_{ij}$ for the algebra
\[ S_{ij} \,=\, \C[s, s\theta_{i,1}, s \theta_{i,2}, s \theta_{i,3}, s\theta_{i,4}, s\theta_{j,1}, s \theta_{j,2}, s \theta_{j,3}, s\theta_{j,4}] \subset \C[s,u_i,v_i,u_j,v_j].\]
\begin{lemma} \label{lem:Khov1}
The polynomials $s\theta_0, s\theta_{i,1}, \ldots, s\theta_{i,4}, s\theta_{j,1}, \ldots,s\theta_{j,4}$ are a Khovanskii basis for the graded algebra $S_{ij}$ with respect to the monomial order \eqref{eq:monorder} restricted to $\C[s,u_i,v_i,u_j,v_j]$.
\end{lemma}
\begin{proof}
 In the previous section, we computed that the 9 leading monomials in $u_i, v_i, u_j, v_j$ of the generators of $S_{ij}$ define a toric ideal with 10 generators. The generators are listed in Equation \eqref{eq:toricidealij}. By \cite[Theorem 2.17]{Kaveh-Manon}, it suffices to substitute $y_0 = h_0$ and $y_{i,j} = h_{i,j}$ into the binomials \eqref{eq:toricidealij}, and show that the resulting 10 polynomials $\varphi_1, \ldots, \varphi_{10} \in \C[s, u_i, v_i, u_j, v_j]$ lie in $S_{ij}$. We have
\[ \begin{array}{lll}
 \varphi_1 = \varphi_2 = 0, &
 \varphi_3 = \theta_{i,1}^2 \theta_{i,2}, &
 \varphi_4 = \theta_{j,1}^2\theta_{j,2}, \\[0.3em]
 \varphi_5 = \theta_{i,1}^3, &
 \varphi_6 = \theta_{j,1}^3, &
 \varphi_7 = \theta_{i,1}^3\theta_{i,3}, \\[0.3em]
 \varphi_8 = \theta_{j,1}^3\theta_{j,3},&
 \varphi_9 = \theta_{i,3}^2\theta_{j,1}^3\theta_{j,3} - \theta_{i,1}^3\theta_{i,3}\theta_{j,3}^2, &
 \varphi_{10} =  \theta_{i,3}^3\theta_{j,1}^3\theta_{j,3} - \theta_{i,1}^3 \theta_{i,3}^2\theta_{j,3}^2,
 \end{array}\]
and one checks directly that they are all in $S_{i,j}$
Here $\varphi_1, \ldots, \varphi_5$ come from the left column of~\eqref{eq:toricidealij}, and $\varphi_6, \ldots, \varphi_{10}$ from the right column.
\end{proof}
\begin{remark}
As an alternative to the above proof, one can check that $h_0, h_{i,j}$ form a Khovanskii basis for $S_{ij}$ using the Macaulay2 \cite{M2} package \texttt{SubalgebraBases}.
\end{remark}

Combining Theorem \ref{thm:2elemsubsets} and Lemma \ref{lem:Khov1}, we arrive at the following result.

\begin{corollary} \label{cor:Khov1}
The polynomials $s h_0, sh_{i,j}$ from \eqref{eq:paramY} form a Khovanskii basis for $S_\phi$ with respect to the monomial order $\eqref{eq:monorder}$.
\end{corollary}
\begin{proof}
This follows from Theorem \ref{thm:2elemsubsets}, Lemma \ref{lem:Khov1} and \cite[Theorem 2.17]{Kaveh-Manon}.
\end{proof}

We can now finally prove our main theorem.
\begin{proof}[proof of Theorem \ref{thm:main}]
By Lemmas \ref{lem:lowerbound} and \ref{lem:upperbound}, we have $5^N \leq \delta_N \leq \deg(Y_N)$. By Corollary \ref{cor:Khov1}, the hypotheses of Theorem \ref{thm:mainkhov} are satisfied. Therefore, combining Theorem \ref{thm:mainkhov} and Lemma \ref{lem:volQN}, we conclude $\deg(Y_N) = {\rm vol}(Q_N) = 5^N$.
\end{proof}

We conclude the section with a generalization of Theorem \ref{thm:main} which allows to restrict the parameters to certain subvarieties of the parameter space, without altering the generic number of solutions. The general statement is that almost all parameter choices in an irreducible subvariety $V \subset \C^{2(N+3)N}$, not contained in the discriminant variety $\Delta$ from the Introduction, still gives $5^N$ solutions. The following corollary identifies some natural candidates for $V$.
\begin{corollary}
Let $V \subset \C^{2(N+3)N}$ be any irreducible subvariety containing the linear space $\{ c_{j,i} = 0, d_{j,i} = 0 \}$ where all parameters $c$ and $d$ vanish. Then, for almost all choices of parameters $ (a_i,b_i,c_i,d_i) \mid i = 1, \ldots, N)  \in V$, the number of complex solutions to \eqref{eq:N-sec5} is $5^N$.
\end{corollary}
\begin{proof}
The discriminant variety $\Delta \subset  \C^{2(N+3)N}$ contains all parameter values for which the number of solutions is not $\delta_N = 5^N$. By our observation in the proof of Lemma \ref{lem:lowerbound}, $V \setminus \Delta$ is non-empty. By irreducibility, $V \cap \Delta$ is a subvariety of $V$ whose complement is dense in $V$.
\end{proof}
The meaning of this corollary is that, if we put constraints in the $c$ and $d$ parameters given by some polynomial equations $\varphi(c,d) = 0$ satisfying $\varphi(0,0)=0$, then almost all choices of parameters under this constraint still give $5^N$ solutions. For instance, $\varphi$ could impose some of the $(c,d)$ variables to be equal to each other, or to be equal to zero.

%

\section{Numerical nonlinear algebra} \label{sec:6}

We use \cref{thm:main} for computations. Recall from \cref{eq:N} that $\mathcal F_N=0$ is a system of~$2N$ equations in the $2N$ variables $(u,v) = (u_1,\ldots,u_N,v_1,\ldots,v_N)$ depending on $2(N+3)N$ parameters $p:=\big((a_i,b_i,c_i,d_i)\mid  i= 1,\ldots, N\big)\in\mathbb C^{2(N+3)N}$. To emphasize these dependencies let us write in this section
$$\mathcal F_N(u,v; p) =
\begin{bmatrix}a_{1,i} u_i \, (u_i^2+v_i^2) + a_{2,i} \, u_i + a_{3,i} \, v_i + a_{4,i} + \sum_{j\ne i} c_{j,i}v_j,\\
b_{1,i} \, v_i \, (u_i^2+v_i^2) + b_{2,i} \, u_i \, + b_{3,i} \, v_i \, + b_{4,i} +  \sum_{j\ne i} d_{j,i}u_j
\end{bmatrix}_{1\leq i \leq N}.
$$

By \cref{thm:main}, for almost all choices of parameters $p\in\mathbb C^{2(N+3)N}$ the number of complex solutions of $\mathcal F_N(u,v; p)=0$ is
$5^N.$ More specifically, by \cref{lem:LSC}, the parameters for which this is not the case are contained in a proper algebraic subvariety $\Delta\subsetneq\mathbb C^{2(N+3)N}$ (which we called the discriminant) and for these parameters the number of isolated solutions can only decrease.

\begin{table}
\begin{center}
\underline{Experiment 1: Algorithm \ref{alg1}}

\smallskip

\begin{tabular}{|llllll|}
\hline
$N$ & $5^N$ && setting up start system & solving target system & success rate\\
\hline
$2$ & $25$ && $0.26\,\mathrm{s}$ & $0.006\,\mathrm{s}$ & $100\%$
\\
$3$ & $125$ && $0.26\,\mathrm{s}$ & $0.14\,\mathrm{s}$& $100\%$
\\
$4$ & $625$ && $0.26\,\mathrm{s}$ & $1.8\,\mathrm{s}$& $100\%$
\\
$5$ & $3125$ && $0.26\,\mathrm{s}$ & $14\,\mathrm{s}$&$100\%$
\\
$6$ & $15625$ && $0.27\,\mathrm{s}$ & $105.37\,\mathrm{s}$&$100\%$
\\
\hline
\end{tabular}

\bigskip

\underline{Experiment 2: Algorithm \ref{alg2}}

\smallskip

\begin{tabular}{|llllll|}
\hline
$N$ & $5^N$ && setting up start system & solving target system & success rate\\
\hline
$2$ & $25$ && $0.24\,\mathrm{s}$ & $0.006\,\mathrm{s}$ & $92\%$
\\
$3$ & $125$ && $0.25\,\mathrm{s}$ & $0.04\,\mathrm{s}$ & $98\%$
\\
$4$ & $625$ && $0.25\,\mathrm{s}$ & $0.6\,\mathrm{s}$& $98\%$
\\
$5$ & $3125$ && $0.25\,\mathrm{s}$ & $3.76\,\mathrm{s}$&$99\%$
\\
$6$ & $15625$ && $0.34\,\mathrm{s}$ & $54.3\,\mathrm{s}$&$99\%$
\\
\hline
\end{tabular}
\end{center}
\caption{We solve $\mathcal F_N((u,v); p)=0$ using Algorithms \ref{alg1} and \ref{alg2}, where $p\in\mathbb C^{2(N+3)N}$ is chosen with independent standard normal entries. The first table is for Algorithm \ref{alg1}. The second is for Algorithm \ref{alg2}. The tables show the values of $N$, corresponding computations times, and how many zeros were successfully computed.\label{table}.}
\end{table}

Fix parameters $p \in \C^{2(N+3)N}$. We want to find \emph{all} solutions of ${\mathcal F}_N(u,v;p) = 0$. The idea of homotopy continuation is to pick different parameters $p_0 \in\mathbb C^{2(N+3)N}\setminus \Delta$, such that we can \emph{easily} compute the solutions of ${\mathcal F}_N(u,v;p_0)=0$. Then, we connect $p_0$ and $p$ by a smooth path $\gamma(t)$ with $\gamma(1)=p_0$ and $\gamma(0)=p$.
This defines the \emph{homotopy} $H((u,v),t) = \mathcal F_N(u,v; \gamma(t))$. The solutions describe continuous paths $u(t), v(t)$ for $t$ going from 1 to 0, which satisfy $H((u(t),v(t)),t) = 0$. Taking the total derivative with respect to $t$ gives us the ODE
\[ \frac{\partial H((u,v),t)}{\partial t}  + \frac{\mathrm d H((u,v),t) }{\mathrm d (u,v)} \cdot \begin{bmatrix}\dot u \\ \dot v\end{bmatrix} = 0. \]
Here $\tfrac{\mathrm d H((u,v),t) }{\mathrm d (u,v)}$ is the Jacobian matrix of $H$ with respect to the variables $u,v$. By construction, we have at $t=1$ that $H((u,v),1)=\mathcal F_N((u,v); p_0)$. Any solution $(u^*,v^*)$ of $\mathcal F_N((u,v); p_0)=0$ yields an initial value problem with initial conditions $u(1) = u^*, v(1) = v^*$ that can be solved using standard numerical ODE methods. We also say that we \emph{track} the \emph{start solution} $(u^*,v^*)$ along the path $\gamma(t)$ when we solve this initial value problem for $u(t), v(t)$. The desired \emph{target solution} is $u(0), v(0)$.
Since $\Delta$ is of complex codimension (at least) 1, it is of real codimension (at least)~2, which means that we can always find a path $\gamma(t)$ that does not intersect $\Delta$ for~$0<t\leq 1$. For such paths, the above initial value problems are well-posed\footnote{Well-posed except at maybe $t=0$. If the problem is ill-posed at $t=0$, this means $p \in \Delta$ and one can use so-called \emph{endgames}; see~\cite[Section 10]{sommese2005numerical}.}. For more details, see the textbook by Sommese and Wampler \cite{sommese2005numerical}.

One way to choose the path $\gamma(t)$ comes from the toric degeneration induced by our monomial order \eqref{eq:monorder}. This was worked out by Burr, Sottile and Walker in \cite{burr2020numerical}. Their method computes the solutions of the \emph{start system} $\mathcal F_N((u,v); p_0)=0$ via a \emph{polyhedral homotopy}, which relies on polyhedral computations in dimension $2N$. This can be quite demanding for increasing $N$.

In the setting of our paper we obtain the following more efficient way for setting up the homotopy. Theorem \ref{thm:mainN=1} implies that general parameters of the form
$$p_0=\big((a_i,b_i,0,0)\mid  i= 1,\ldots, N\big)\in\mathbb C^{2(N+3)N},$$
i.e., general parameters such that the coupling coefficients are zero, yield $5^N$ solutions. Notice that these $5^N$ solutions are given as a cartesian product of $N$ sets of $5$ solutions. Consequently, we can set up the start system and compute its solutions by solving separately $N$ systems of $2$ equations in $2$ variables. The complexity is \emph{linear} in $N$! We use this strategy for a homotopy continuation algorithm that we call \emph{Decoupled Homotopy}. It is summarized in Algorithm \ref{alg1}.

\begin{algorithm}
\caption{Decoupled Homotopy\label{alg1}}
\SetAlgoLined
\KwIn{A list of parameters $p \in\mathbb C^{2(N+3)N}$.}
\KwOut{The complex zeros of $\mathcal F_N((u,v);p)$.}
Sample $(a_i,b_i,c_i,d_i \mid i=1,\ldots,N)\in \mathbb C^{2(N+3)N}$ by choosing independently real and imaginary parts from the standard normal distribution\;
\For{$i=1,\ldots,N$}{
Solve
$$\begin{bmatrix}
\, a_{1,i} u_i \, (u_i^2+v_i^2) + a_{2,i} \, u_i + a_{3,i} \, v_i + a_{4,i}\, \\
\, b_{1,i} \, v_i \, (u_i^2+v_i^2) + b_{2,i} \, u_i \, + b_{3,i} \, v_i \, + b_{4,i}\,
\end{bmatrix} =0,$$
and let $s_{i,1},\ldots,s_{i,5}$ be the 5 computed solutions\;
}
Set $S_0:=\{(s_{i_1,1}, \ldots, s_{i_5,5})\mid (i_1,\ldots,i_N)\in\{1,\ldots,5\}^{\times N}\}$\;
Set $p_0:=(a_i,b_i,0,0\mid i =1,\ldots,N)$\;
Set $p_1:=(a_i,b_i,c_i,d_i\mid i =1,\ldots,N)$\;
\For{$s_0 \in S_0$}{
Use homotopy continuation to track $s_0$
from $p_0$ to $p_1$ along $\gamma(t)=tp_0+(1-t)p_1$.
Let $s_1$ be the computed solution\;
Use homotopy continuation to track $s_1$
from $p_1$ to $p$ along $\gamma(t) = tp_1+(1-t)p$.\quad
Let $s$ be the computed solution\;
Return $s$\;
}
\end{algorithm}

Due to step 1 of Algorithm \ref{alg1}, $p_0$ is general with probability one. Furthermore, the two homotopies in steps 11 \& 12 make sure that the (piecewise) linear path from $p_0$ to $p$ does not intersect the discriminant $\Delta$ with probability one (see \cite[Section 7]{sommese2005numerical}).

We use \texttt{HomotopyContinuation.jl} \cite{HC.jl} to experiment with Algorithm \ref{alg1}. The code is attached to the \texttt{ArXiv} version of this article. For $2\leq N\leq 6$ we choose \emph{real} parameters $p\in\mathbb R^{2(N+3)N}$ by sampling the entries independently from the standard normal distribution. Then, we apply Algorithm \ref{alg1} to solve $\mathcal F_N((u,v);p)=0$. The computation times in \Cref{table} were obtained on a laptop with 2,3 GHz Dual-Core Intel Core i5 and 8 GB Memory. The time it takes to set up the start system increases minimally, because the computational complexity of computing the start system is linear in $N$. This shows the strength of the Decoupled Homotopy.

Algorithm \ref{alg1} works correctly with probability one. We think of an alternative algorithm: Consider parameters $p=(a_i,b_i,c_i,d_i \mid i=1,\ldots,N)\in \mathbb C^{2(N+3)N}$. Instead of sampling the start parameter $p_0$ and the middle parameter $p_1$, we take $p_0=(a_i,b_i,0,0 \mid i=1,\ldots,N)\in \mathbb C^{2(N+3)N}$. That is, we obtain the starting parameters $p_0$ from the target parameters $p$ by simply setting the $c$ and $d$ coordinates to $0$. We then track the solutions of $\mathcal F_N((u,v); p_0)=0$ along the path $tp_0 + (1-t)p$. This gives Algorithm \ref{alg2}.
The last column of \Cref{table} indicates that Algorithm \ref{alg2} can't compute all zeros of $\mathcal F_N((u,v); p)$, but it outperforms Algorithm~\ref{alg1} in terms of computation times.

\begin{algorithm}
\caption{Decoupled Homotopy (Heuristic Version)\label{alg2}}
\SetAlgoLined
\KwIn{A list of parameters $p=(a_i,b_i,c_i,d_i \mid i=1,\ldots,N)\in \mathbb C^{2(N+3)N}$.}
\KwOut{The complex zeros of $\mathcal F_N((u,v);p)$.}
\For{$i=1,\ldots,N$}{
Solve
$$\begin{bmatrix}
\, a_{1,i} u_i \, (u_i^2+v_i^2) + a_{2,i} \, u_i + a_{3,i} \, v_i + a_{4,i}\, \\
\, b_{1,i} \, v_i \, (u_i^2+v_i^2) + b_{2,i} \, u_i \, + b_{3,i} \, v_i \, + b_{4,i}\,
\end{bmatrix} =0,$$
and let $s_{i,1},\ldots,s_{i,5}$ be the 5 computed solutions\;
}
Set $S_0:=\{(s_{i_1,1}, \ldots, s_{i_5,5})\mid (i_1,\ldots,i_N)\in\{1,\ldots,5\}^{\times N}\}$\;
Set $p_0:=(a_i,b_i,0,0\mid i =1,\ldots,N)$\;
\For{$s_0 \in S_0$}{
Use homotopy continuation to track $s_0$
from $p_0$ to $p$ along $\gamma(t) = tp_0+(1-t)p$.\quad Let $s$ be the computed solution\;
Return $s$\;
}
\end{algorithm}


\printbibliography

\end{document}